\documentclass[a4paper,12pt]{article}
\usepackage{amsmath}
\usepackage{amssymb}
\usepackage{amsthm}
\usepackage{latexsym}
\usepackage{amsmath}
\usepackage{amsthm}
\usepackage[dvipdfmx]{graphicx}
\usepackage{txfonts}
\usepackage{bm}
\usepackage{color}
\usepackage{epic,eepic}

\usepackage{geometry}
\numberwithin{equation}{section}

\newtheorem{theorem}{Theorem}[section]
\newtheorem{proposition}[theorem]{Proposition}
\newtheorem{corollary}[theorem]{Corollary}
\newtheorem{lemma}[theorem]{Lemma}
\newtheorem{remark}{Remark}[section]
\newtheorem{example}{Example}[section]

\newcommand{\todaye}{\the\year/\the\month/\the\day}
\newcommand{\finbox}{\hspace*{\fill}$\rule{0.2cm}{0.2cm}$}

\newcommand{\RR}{{\mathbb{R}}}
\newcommand{\ZZ}{{\mathbb{Z}}}

\newcommand{\veczero}{{\bf 0}}
\newcommand{\dom}{{\rm dom\,}}

\newcommand{\suppp}{{\rm supp}^{+}}
\newcommand{\suppm}{{\rm supp}^{-}}
\newcommand{\unitvec}[1]{\bm{1}_{#1}}

\newcommand{\argmin}{\arg \min}

\title{Scaling, Proximity, and Optimization of 
\\ Integrally Convex Functions%
\footnote{The extended abstract of this paper is included in the Proceedings
of the 27th International Symposium on Algorithms and Computation
(ISAAC), Sydney, December 12--14, 2016.
Leibniz International Proceedings in Informatics (LIPIcs),
{\bf 64} (2016), 57:1--57:12,
Dagstuhl Publishing.
}
}

\author{
Satoko Moriguchi%
\thanks{Tokyo Metropolitan University, satoko5@tmu.ac.jp},
Kazuo Murota%
\thanks{Tokyo Metropolitan University, murota@tmu.ac.jp}, 
Akihisa Tamura%
\thanks{Keio University, aki-tamura@math.keio.ac.jp},
Fabio Tardella%
\thanks{Sapienza University of Rome, fabio.tardella@uniroma1.it}
}

\date{March 29, 2017 (Revised December 12, 2017)}
\begin{document}

\maketitle

\begin{abstract}
In discrete convex analysis, the scaling and proximity properties for the class 
of L$^{\natural}$-convex functions were established 
more than a decade ago and have been used to design efficient minimization algorithms. 
For the larger class of integrally convex functions of $n$ variables, 
we show here that the scaling property only holds when 
$n \leq 2$, while a proximity theorem
can be established for any $n$, but only with a superexponential bound.
This is, however, sufficient to extend the classical
logarithmic complexity result for minimizing a discrete convex function
of one variable to the case of integrally convex functions
of any fixed number of variables.
\end{abstract}

\section{Introduction}
\label{SCintro}

The proximity-scaling approach is a fundamental technique
in designing efficient algorithms for discrete or combinatorial optimization.
For a function
$f: \mathbb{Z}^{n} \to \mathbb{R} \cup \{ +\infty  \}$
in integer variables and a positive integer $\alpha$, called 
a scaling unit, the $\alpha$-scaling of $f$ means the function 
$f^{\alpha}$
defined by $f^{\alpha}(x) = f(\alpha x) $ $(x \in \mathbb{Z}^{n})$.
A proximity theorem is a result guaranteeing that a (local) minimum 
of the scaled function $f^{\alpha}$
is close to a minimizer 
of the original function $f$. 
The scaled function $f^{\alpha}$ is simpler
and hence easier to minimize,
whereas the quality of the obtained minimizer of $f^{\alpha}$ 
as an approximation to the minimizer of $f$
is guaranteed by a proximity theorem.
The proximity-scaling approach consists in applying this idea
for a decreasing sequence of $\alpha$, often by halving the scale unit $\alpha$.
A generic form of a proximity-scaling algorithm may be described as follows,
where $K_{\infty}\ (> 0)$ denotes the $\ell_{\infty}$-size of 
the effective domain
$\dom f = \{ x \in \ZZ^{n} \mid f(x) < +\infty \}$ and $B(n,\alpha)$ 
denotes the proximity bound in $\ell_{\infty}$-distance for $f^{\alpha}$.

\begin{tabbing}     
\= {\bf Proximity-scaling algorithm}%
\\
\> \quad  S0: 
   \= Find an initial vector $x$ with $f(x) < +\infty$, and set 
   $\alpha := 2\sp{\lceil \log_{2} K_{\infty} \rceil}$.
\\
\> \quad  S1:
   \>  Find an integer vector $y$ with   $\| \alpha y \|_{\infty} \leq B(n,\alpha)$
    that is a (local) minimizer of  
\\ \> \>
      $\tilde f(y) = f(x + \alpha y)$,
    and set $x:= x+ \alpha y$.  \\
\> \quad  S2: 
    \> If $\alpha = 1$, then stop \
       ($x$ is a minimizer of $f$).             
\\
\> \quad  S3: 
  \> Set  $\alpha:=\alpha/2$, and go to S1.  
\end{tabbing}
The algorithm consists of $O(\log_{2} K_{\infty})$ scaling phases.
This approach has been particularly successful 
for resource allocation problems \cite{Hoc07,HS90,IK88,KSI13}
and for convex network flow problems
(under the name of ``capacity scaling'') \cite{AMO93,IMM05submflow,IS02}.
Different types of proximity theorems have also been investigated: 
proximity between integral and real optimal solutions 
\cite{HS90,Tami93,Tami09}, among others.
For other types of algorithms of nonlinear integer optimization,
see, e.g., \cite{HKLW10}. 

In discrete convex analysis \cite{Mdca98,Mdcasiam,Mbonn09,Mdcaeco16},
a variety of discrete convex functions are considered.
A {\em separable convex} function is a function
$f: \ZZ^{n} \to \RR \cup \{ +\infty \}$
that can be represented as
$f(x) = \varphi_{1}(x_{1}) + \cdots + \varphi_{n}(x_{n})$,
where $x=(x_{1}, \ldots,x_{n})$,
with univariate discrete convex functions
$\varphi_{i}: \ZZ \to \RR \cup \{ +\infty \}$
satisfying 
$\varphi_{i}(t-1) + \varphi_{i}(t+1) \geq 2 \varphi_{i}(t)$
for all $t \in \ZZ$.

A function
$f: \ZZ^{n} \to \RR \cup \{ +\infty \}$
is called {\em integrally convex}
if its local convex extension 
$\tilde{f}: \RR^{n} \to \RR \cup \{ +\infty \}$  
is (globally) convex in the ordinary sense, where
$\tilde{f}$
is defined as the collection of convex extensions of $f$ in each 
unit hypercube 
$\{ x \in \RR\sp{n} \mid a_{i} \leq x_{i} \leq a_{i} + 1 \ (i=1,\ldots, n) \}$ 
with $a \in \ZZ^{n}$;
see Section~\ref{SCintcnvfn} for precise statements.

A function
$f: \ZZ^{n} \to \RR \cup \{ +\infty \}$
is called
{\em L$^{\natural}$-convex}
if it satisfies 
one of the equivalent conditions in 
Theorem~\ref{THlnatcond} below.
For $x,y \in \ZZ^{n}$, 
$x \vee y$ and $x \wedge y$ denote
the vectors of componentwise maximum and minimum of $x$ and $y$, respectively.
Discrete midpoint convexity of $f$ for $x,y \in \ZZ^{n}$ means 
\begin{equation} \label{lnatfmidconv}
 f(x) + f(y) \geq
    f \left(\left\lceil \frac{x+y}{2} \right\rceil\right) 
  + f \left(\left\lfloor \frac{x+y}{2} \right\rfloor\right) ,
\end{equation}
where $\lceil \cdot \rceil$ and $\lfloor \cdot \rfloor$ denote
the integer vectors obtained by
componentwise rounding-up and rounding-down to the nearest integers,
respectively.
We use the notation $\bm{1}=(1,1,\ldots, 1)$
and 
$\unitvec{i}$ for the $i$-th unit vector 
$(0,\ldots,0, \overset{\overset{i}{\vee}}1,0,\ldots,0)$,
with the convention $\unitvec{0}=\bm{0}$.

\begin{theorem}
[\protect{\cite{FT90, FM00,Mdcasiam}}]  \label{THlnatcond}
For $f: \ZZ^{n} \to \RR \cup \{ +\infty \}$
the following conditions, {\rm (a)} to {\rm (d)}, are equivalent:%
\footnote{
$\ZZ$-valued functions are treated in \cite[Theorem~3]{FM00},
but the proof is valid for $\RR$-valued functions.
} 

{\rm (a)}
$f$ is integrally convex and submodular: 
\begin{equation} \label{lnatfsubm}
 f(x) + f(y) \geq f(x \vee y) + f(x \wedge y) 
 \qquad  (x, y \in \ZZ^{n}). 
\end{equation}

{\rm (b)}
$f$ satisfies 
discrete midpoint convexity 
{\rm (\ref{lnatfmidconv})} for all $ x, y \in \ZZ^{n}$.

{\rm (c)}
$f$ satisfies 
discrete midpoint convexity 
{\rm (\ref{lnatfmidconv})} for all $ x, y \in \ZZ^{n}$
with $\| x-y \|_{\infty} \leq 2$, and 
the effective domain has the property:
$x, y \in \dom f   \Rightarrow   
   \left\lceil  (x+y)/2 \right\rceil ,
   \left\lfloor (x+y)/2 \right\rfloor \in \dom f $.

{\rm (d)}
$f$ satisfies translation-submodularity:
\begin{equation} \label{lnatftrsubm}
  f(x) + f(y) \geq f((x - \mu {\bf 1}) \vee y) 
                 + f(x \wedge (y + \mu {\bf 1}))
\qquad  (x, y \in \ZZ^{n}, \ 0 \leq \mu \in \ZZ).
\end{equation}
\vspace{-1.7\baselineskip}
\\
\finbox
\end{theorem}

A simple example to illustrate the difference between integrally convex and
L$^{\natural}$-convex functions can be provided in the case of quadratic
functions. Indeed, for an $n \times n$ symmetric matrix $Q$ and a vector
$p \in \RR^n$, the function $f(x) = x\sp{\top} Q x + p\sp{\top} x$ is integrally
convex whenever $Q$ is diagonally dominant with nonnegative diagonal elements, i.e.,
$q_{ii} \geq \sum_{j \not= i} |q_{ij}|$ for $i=1,\ldots,n$ \cite{FT90}.
On the other hand, $f$ is L$^{\natural}$-convex if and only if it is
diagonally dominant with nonnegative diagonal elements and $q_{ij} \leq 0$
for all $i \neq j$ \cite[Section 7.3]{Mdcasiam}.

A function
$f: \ZZ^{n} \to \RR \cup \{ +\infty \}$
is called
{\em M\/$^{\natural}$-convex}
if it satisfies an exchange property:
For any $x, y \in \dom f$ and any $i \in \suppp(x-y)$,
there exists $j \in \suppm(x-y) \cup \{ 0 \}$ such that  
\begin{equation} \label{mnatfnexc}
 f(x)+f(y) \geq 
  f(x - \unitvec{i} + \unitvec{j}) 
             + f(y + \unitvec{i} - \unitvec{j}) ,
\end{equation}
where, for $z \in \ZZ^{n}$, 
$\suppp(z) = \{ i \mid z_{i} > 0 \}$ and
$\suppm(z) = \{ j \mid z_{j} < 0 \}$.
It is known (and easy to see) that a function is separable convex if and only if it is both
${\rm L}^{\natural}$-convex and ${\rm M}^{\natural}$-convex.

Integrally convex functions 
constitute a common framework for discrete convex functions,
including separable convex,
${\rm L}^{\natural}$-convex and
${\rm M}^{\natural}$-convex functions  
as well as 
${\rm L}^{\natural}_{2}$-convex and 
${\rm M}^{\natural}_{2}$-convex functions \cite{Mdcasiam},
and BS-convex and UJ-convex functions \cite{Fuj14bisubmdc}.
The concept of integral convexity is
used in formulating discrete fixed point theorems
\cite{Iim10,IMT05,Yan09fixpt}, 
and designing solution algorithms for discrete systems of nonlinear equations
\cite{LTY11nle,Yan08comp}.
In game theory 
the integral concavity of payoff functions 
guarantees the existence of a pure strategy equilibrium 
in finite symmetric games \cite{IW14}.

The scaling operation preserves ${\rm L}^{\natural}$-convexity,
that is,
if $f$ is ${\rm L}^{\natural}$-convex, then 
$f^{\alpha}$ is ${\rm L}^{\natural}$-convex.
${\rm M}^{\natural}$-convexity is subtle in this respect:
for an ${\rm M}^{\natural}$-convex function $f$, 
$f^{\alpha}$ remains ${\rm M}^{\natural}$-convex if $n \leq 2$, 
while this is not always the case if $n\geq 3$.

\begin{example} \rm \label{EXmnatdim3}
Here is an example to show that
${\rm M}^{\natural}$-convexity
is not preserved under scaling.
Let $f$ be the indicator function of the set 
$S =  \{ c_{1} (1,0,-1) +c_{2} (1,0,0) + c_{3} (0,1,-1) + c_{4} (0,1,0)
\mid c_{i} \in \{ 0,1 \} \ (i=1,2,3,4)  \} \subseteq \ZZ^{3}$.
Then $f$ is an ${\rm M}^{\natural}$-convex function,
but $f^{2}$ (=$f^{\alpha}$ with $\alpha=2$),
being the indicator function of $ \{ (0,0,0), (1,1,-1) \}$,
is not ${\rm M}^{\natural}$-convex.
This example is a reformulation of \cite[Note 6.18]{Mdcasiam}
for {\rm M}-convex functions to ${\rm M}^{\natural}$-convex functions.
\finbox
\end{example}

It is rather surprising that nothing is known about scaling for
integrally convex functions.
Example~\ref{EXmnatdim3} does not demonstrate the lack of scaling property 
of integrally convex functions, 
since $f^{2}$ above is integrally convex, though not ${\rm M}^{\natural}$-convex.

As for proximity theorems, the following facts are known for separable convex,
 ${\rm L}^{\natural}$-convex and ${\rm M}^{\natural}$-convex functions.
In the following three theorems we assume that 
$f: \mathbb{Z}^{n} \to \mathbb{R} \cup \{ +\infty  \}$, 
$\alpha$ is a positive integer, and $x^{\alpha} \in \dom f$.
It is noteworthy that the proximity bound is
independent of $n$ for separable convex functions, and 
coincides with  $n (\alpha -1)$, which is linear in $n$,
for ${\rm L}^{\natural}$-convex and ${\rm M}^{\natural}$-convex functions.

\begin{theorem}   \label{THsepfnproximity}
Suppose that $f$ is a separable convex function.
If $f(x^{\alpha}) \leq f(x^{\alpha} +  \alpha d)$ for all 
$d \in \{ \unitvec{i}, -\unitvec{i}  \  (1 \leq i \leq n) \} $,
then there exists a minimizer 
$x^{*}$ of $f$ with $\| x^{\alpha} - x^{*}  \|_{\infty} \leq \alpha -1$.
\end{theorem}

\begin{proof}
The statement is obviously true if $n=1$.
Then the statement for general $n$ follows easily from the fact that
$x^{*}$ is a minimizer of 
$f(x) = \varphi_{1}(x_{1}) + \cdots + \varphi_{n}(x_{n})$
if and only if, for each $i$,  $x^{*}_{i}$ is a minimizer of $\varphi_{i}$.
\end{proof}

\begin{theorem} [\protect{\cite{IS02}}; {\cite[Theorem 7.18]{Mdcasiam}}]  \label{THlfnproximity} 
Suppose that $f$ is an ${\rm L}^{\natural}$-convex function.
If $f(x^{\alpha}) \leq f(x^{\alpha} +  \alpha d)$ for all 
$d \in  \{ 0, 1 \}^{n} \cup \{ 0, -1 \}^{n}$,
then there exists a minimizer 
$x^{*}$ of $f$ with $\| x^{\alpha} - x^{*}  \|_{\infty} \leq n (\alpha -1)$.
\finbox
\end{theorem}

\begin{theorem}[\protect{\cite{MMS02}}; \protect{\cite[Theorem 6.37]{Mdcasiam}}]\label{THmfnproximity}
Suppose that $f$ is an ${\rm M}^{\natural}$-convex function.
If $f(x^{\alpha}) \leq f(x^{\alpha} + \alpha d)$ for all 
$d \in \{ \unitvec{i}, -\unitvec{i}  \  (1 \leq i \leq n),
 \  \unitvec{i} - \unitvec{j} \  (i \not= j) \} $,
then there exists a minimizer 
$x^{*}$ of $f$ with $\| x^{\alpha} - x^{*}  \|_{\infty} \leq n (\alpha -1)$.
\finbox
\end{theorem}

Based on the above results and their variants, efficient algorithms for minimizing
${\rm L}^{\natural}$-convex and ${\rm M}^{\natural}$-convex functions
have been successfully designed with the proximity-scaling approach
 (\cite{MMS02,MST11Mrelax,MT09Lrelax,Mdcasiam,Shimin04,Tam05scale}).
Proximity theorems are also available for 
${\rm L}^{\natural}_{2}$-convex and 
${\rm M}^{\natural}_{2}$-convex functions \cite{MT04prox}
and L-convex functions on graphs \cite{Hir15Lext,Hir16Lgraph}.
However, no proximity theorem has yet been proved for integrally convex functions.

\bigskip

The new findings of this paper are

\begin{itemize}
\item  
A ``box-barrier property'' (Theorem~\ref{THintcnvbox}), which allows us to 
restrict the search for a global minimum of an integrally convex function;

\item  
Stability of integral convexity under scaling when $n = 2$
(Theorem~\ref{THintcnvscdim2}),
and an example to demonstrate its failure 
when $n \geq 3$ (Example~\ref{EXscalingNG422});

\item  
A proximity theorem with a superexponential bound
 $\displaystyle  [(n+1)!/ 2^{n-1}](\alpha -1)$
for all $n$ (Theorem~\ref{THproxintcnv}),
and the impossibility of finding a proximity bound 
of the form $B(n)(\alpha -1)$ 
where $B(n)$ is linear or smaller than quadratic
(Examples~\ref{EXproxNG422} and \ref{EXproxtigntM2}).
\end{itemize}

\smallskip

As a consequence of our proximity and scaling results, we derive that:

\begin{itemize}
\item
When $n$ is fixed,
an integrally convex function  
can be minimized in $O(\log_{2} K_{\infty})$ time
by standard proximity-scaling algorithms, 
where 
$K_{\infty} = \max\{ \| x -y \|_{\infty} \mid x, y \in \dom f \}$ 
denotes the $\ell_{\infty}$-size of $\dom f$.
\end{itemize}

This paper is organized as follows.
In Section~\ref{SCintcnvfn} the concept of integrally convex functions is reviewed 
with some new observations and, in Section~\ref{SCscalingIC},
their scaling property is clarified.
After a preliminary discussion in Section~\ref{SCpredisproxIC},
a proximity theorem for integrally convex functions 
is established in Section~\ref{SCproxthmICgeneral}.
Algorithmic implications of the proximity-scaling results are discussed
in Section~\ref{SCalgimpl}
and concluding remarks are made in Section~\ref{SCconrem}.

\section{Integrally Convex Sets and Functions}
\label{SCintcnvfn}

For $x \in \RR^{n}$
the integer neighborhood of $x$ is defined as 
\[
N(x) = \{ z \in \mathbb{Z}^{n} \mid | x_{i} - z_{i} | < 1 \ (i=1,\ldots,n)  \}.
\]
For a function
$f: \mathbb{Z}^{n} \to \mathbb{R} \cup \{ +\infty  \}$
the local convex extension 
$\tilde{f}: \RR^{n} \to \RR \cup \{ +\infty \}$
of $f$ is defined 
as the union of all convex envelopes of $f$ on $N(x)$ as follows:
\begin{equation} \label{fnconvclosureloc2}
 \tilde f(x) = 
  \min\{ \sum_{y \in N(x)} \lambda_{y} f(y) \mid
      \sum_{y \in N(x)} \lambda_{y} y = x,  
  (\lambda_{y})  \in \Lambda(x) \}
\quad (x \in \RR^{n}) ,
\end{equation} 
where $\Lambda(x)$ denotes the set of coefficients for convex combinations indexed by $N(x)$:
\[ 
  \Lambda(x) = \{ (\lambda_{y} \mid y \in N(x) ) \mid 
      \sum_{y \in N(x)} \lambda_{y} = 1, 
      \lambda_{y} \geq 0 \ (\forall y \in N(x))  \} .
\] 
If $\tilde f$ is convex on $\RR^{n}$,
then $f$ is said to be {\em integrally convex} \cite{FT90}.
A set $S \subseteq \ZZ^{n}$ is said to be 
integrally convex if
the convex hull $\overline{S}$ of $S$ coincides with the union of the 
convex hulls of $S \cap N(x)$ over $x \in \RR^{n}$,
i.e., if, for any $x \in \RR^{n}$, 
$x \in \overline{S} $ implies $x \in  \overline{S \cap N(x)}$.
A set $S \subseteq \ZZ^{n}$ is integrally convex if and only if
its indicator function is an integrally convex function.
The effective domain and the set of minimizers of an integrally convex function
are both integrally convex \cite[Proposition 3.28]{Mdcasiam};
in particular, the effective domain and the set of minimizers of 
an L$^{\natural}$- or M$^{\natural}$-convex function are integrally convex.

For $n=2$, integrally convex sets are illustrated in Fig.~\ref{FGintconvset}
and their structure is described in the next proposition.
\begin{figure}\begin{center}
\includegraphics[width=0.85\textwidth,clip]{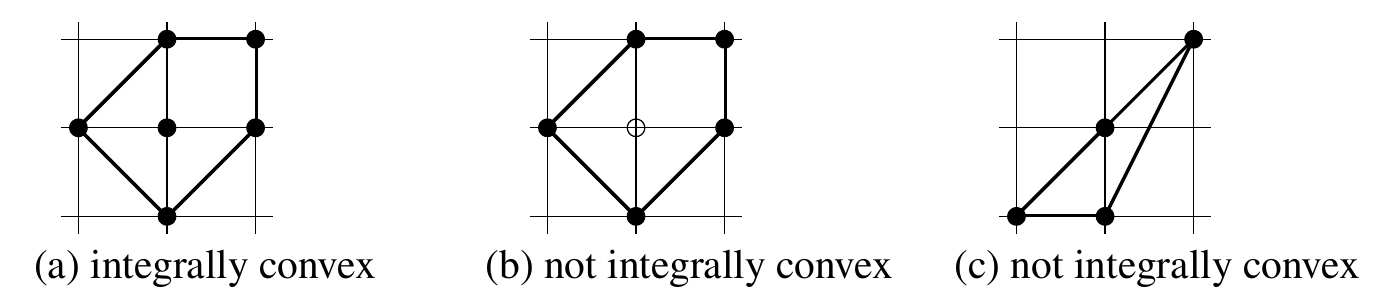}
\caption{Concept of integrally convex sets}
\label{FGintconvset}
\end{center}\end{figure}

\begin{proposition}  \label{PRintcnvsetdim2}
A set $S \subseteq \mathbb{Z}^{2}$ is an integrally convex set
if and only if it can be represented as
$S = \{ (x_{1},x_{2}) \in \mathbb{Z}^{2} \mid
p_{i} x_{1} +  q_{i} x_{2} \leq r_{i} \ (i=1,\ldots,m) \}$
for some $p_{i}, q_{i} \in \{ -1,0,+1 \}$ and $r_{i} \in \ZZ$ $(i=1,\ldots,m)$.
\end{proposition}

\begin{proof}
Consider the convex hull $\overline{S}$ of $S$,
and denote the (shifted) unit square
$\{ (x_{1}, x_{2}) \in \RR\sp{2} \mid a_{i} \leq x_{i} \leq  a_{i}+1 \ (i=1,2) \}$
by $I(a_{1}, a_{2})$,
where $(a_{1}, a_{2}) \in \ZZ\sp{2}$.
Let $S$ be an integrally convex set.
It follows from the definition that
$\overline{S} \cap I(a_{1}, a_{2}) = \overline{S \cap I(a_{1}, a_{2})}$
for each $(a_{1}, a_{2}) \in \ZZ\sp{2}$.
Obviously,
$\overline{S \cap I(a_{1}, a_{2})}$
can be described by (at most four) inequalities
$p'_{j} x_{1} +  q'_{j} x_{2} \leq r'_{j} \ (j=1,\ldots,\ell')$
with $p'_{j}, q'_{j} \in \{ -1,0,+1 \}$ and $r'_{j} \in \ZZ$ $(j=1,\ldots,\ell')$,
where $\ell' = \ell'(a_{1}, a_{2}) \leq 4$.
Since $\overline{S}$ is the union of sets $\overline{S} \cap I(a_{1}, a_{2})$,
$\overline{S}$ can be represented
as $\{ (x_{1},x_{2}) \in \mathbb{R}^{2} \mid
p_{i} x_{1} +  q_{i} x_{2} \leq r_{i} \ (i=1,\ldots,m) \}$
by a subfamily of the inequalities used for all
$\overline{S \cap I(a_{1}, a_{2})}$.
Then we have
$S  = \{ (x_{1},x_{2}) \in \mathbb{Z}^{2} \mid
p_{i} x_{1} +  q_{i} x_{2} \leq r_{i} \ (i=1,\ldots,m) \}$.
Converesly, integral convexity of set $S$ represented in this form for any
$p_{i}, q_{i} \in \{ -1,0,+1 \}$ and $r_{i} \in \ZZ$ is an easy consequence 
of the simple shape of the (possibly unbounded) polygon
$ \{ (x_{1},x_{2}) \in \mathbb{R}^{2} \mid
p_{i} x_{1} +  q_{i} x_{2} \leq r_{i} \ (i=1,\ldots,m) \}$,
which has at most eight edges having directions
parallel to one of the vectors 
$(1,0)$, $(0,1)$, $(1,1)$, $(1,-1)$.
\end{proof}

We note that in the special case where all inequalities
$p_{i} x_{1} +  q_{i} x_{2} \leq r_{i}$  $(i=1,\ldots,m)$
defining $S$ in Proposition~\ref{PRintcnvsetdim2}
satisfy the additional property
$p_{i} q_{i} \leq 0$,
the set $S$ is actually an ${\rm L}^{\natural}$-convex set
\cite[Section 5.5]{Mdcasiam},
which is a special type of sublattice \cite{QT06}.

\begin{remark} \rm \label{RMicsetdim2ineq}
A subtle point in Proposition~\ref{PRintcnvsetdim2} is explained here.
In Proposition~\ref{PRintcnvsetdim2} we do not mean that the system of inequalities for $S$
describes the convex hull $\overline{S}$ of $S$.
That is, it is not claimed that
$\overline{S} = \{ (x_{1},x_{2}) \in \mathbb{R}^{2} \mid
p_{i} x_{1} +  q_{i} x_{2} \leq r_{i} \ (i=1,\ldots,m) \}$
holds.
For instance,
$S = \{ (0,0), (1,0) \}$ is an integrally convex set, which can be represented as
the set of integer points satisfying the four inequalities:
$ -x_{1} +  x_{2} \leq 0$,
$ x_{1} -  x_{2} \leq 1$,
$ x_{1} +  x_{2} \leq 1$, and
$ -x_{1} - x_{2} \leq 0$.
These inequalities, however, do not describe
the convex hull $\overline{S}$,
which is the line segment connecting $(0,0)$ and $(1,0)$.
Nevertheless, it is true in general
(cf. the proof of Proposition~\ref{PRintcnvsetdim2})
that the convex hull of an integrally convex set
can be described by inequalities
of the form of $p'_{i} x_{1} +  q'_{i} x_{2} \leq r'_{i}$
with $p'_{i}, q'_{i} \in \{ -1,0,+1 \}$ and $r'_{i} \in \ZZ$
$(i=1,\ldots,m')$.
For $S = \{ (0,0), (1,0) \}$  we can describe $\overline{S}$
by adding two inequalities
$x_{2} \leq 0$ and $- x_{2} \leq 0$ to the original system of four inequalities.
The present form of Proposition~\ref{PRintcnvsetdim2},
avoiding the convex hull, is convenient
in the proof of Proposition~\ref{PRintcnvsetscdim2}.
\finbox
\end{remark}

\begin{corollary} \label{COintconvset-pairofpoints}
If a set $S \subseteq \ZZ^2$ is integrally convex, then for all points $x,y \in S$,
the set
\begin{align*}
{\rm ICH}(x,y) = &   \{z \in \ZZ^2 \mid  
  \min\{x_i,y_i\} \leq z_i \leq \max\{x_i,y_i\}  \ (i=1,2), \
\\ &
   \min\{x_1-x_2,y_1-y_2\} \leq z_1 - z_2 \leq \max\{x_1-x_2,y_1-y_2\}, \
\\ &
   \min\{x_1+x_2,y_1+y_2\} \leq z_1 + z_2 \leq \max\{x_1+x_2,y_1+y_2\}\  \}
\end{align*}
is contained in S.
\end{corollary}

\begin{proof}
Let $S$ be represented as in Proposition \ref{PRintcnvsetdim2} and let $x,y \in S$.
Then we clearly have  
$ \max \{p_{i} x_{1} +  q_{i} x_{2}, p_{i} y_{1} +  q_{i} y_{2} \} \leq r_{i}$ $(i=1,\ldots,m)$.
The claim follows by observing that 
$\max \{p_{i} x_{1} +  q_{i} x_{2}, p_{i} y_{1} +  q_{i} y_{2} \}$
coincides with one of
$\max \{x_{i}, y_{i} \}$, $\max \{-x_{i}, -y_{i} \}$ $(i=1,2)$, \ 
$\max \{x_{1}-x_{2}, y_{1}- y_{2} \}$, \ 
$\max \{x_{1}+x_{2}, y_{1}+ y_{2} \}$, \  
$\max \{-x_{1}+x_{2}, -y_{1}+ y_{2} \}$, $\max \{-x_{1}-x_{2}, -y_{1}-y_{2} \}$, 
according to the values of $p_{i}, q_{i} \in \{ -1,0,+1 \}$.
\end{proof}

Note that ${\rm ICH}(x,y)$ is integrally convex
by Proposition \ref{PRintcnvsetdim2},
and that, by the above corollary, any integrally convex set 
containing $\{x,y \}$ must contain ${\rm ICH}(x,y)$.
Thus ${\rm ICH}(x,y)$ is the smallest integrally convex set containing $\{x,y \}$.

Integral convexity is  preserved under the operations of
origin shift, permutation of components, and componentwise (individual) sign inversion.
For later reference we state these facts as a proposition.

\begin{proposition}  \label{PRintcnvinvar}
Let $f: \mathbb{Z}^{n} \to \mathbb{R} \cup \{ +\infty  \}$
be an integrally convex function.

\noindent
{\rm (1)} 
For any $z \in \ZZ^{n}$,  $f(z + x)$ is integrally convex in $x$.

\noindent
{\rm (2)} 
For any permutation  $\sigma$ of $(1,2,\ldots,n)$,  
 $f(x_{\sigma(1)}, x_{\sigma(2)}, \ldots, x_{\sigma(n)})$ 
is integrally convex in $x$.
\\
{\rm (3)} 
For any $s_{1}, s_{2}, \ldots, s_{n} \in \{ +1, -1 \}$,
 $f(s_{1} x_{1}, s_{2} x_{2}, \ldots, s_{n} x_{n})$ is integrally convex in $x$.
\end{proposition}

\begin{proof}
The claims (1) to (3) follow easily from the definition of integrally convex 
functions and the obvious relations:
$N(z + x) = 
\{ z + y \mid y \in N(x) \}$,
$ N((x_{\sigma(1)},  \ldots, x_{\sigma(n)})) = 
\{ (y_{\sigma(1)}, \ldots, y_{\sigma(n)}) \mid y \in N(x) \}$,
and $N((s_{1} x_{1},  \ldots, s_{n} x_{n})) =
\{ (s_{1} y_{1},  \ldots, s_{n} y_{n}) \mid y \in N(x) \}$.
\end{proof}

Integral convexity of a function can be characterized by a local condition
under the assumption that the effective domain is an integrally convex set.
The following theorem is proved in \cite{FT90} when the effective domain
is an integer interval (discrete rectangle).
An alternative proof, which is also valid for the general case,
 is given in Appendix \ref{SCintcnvD2proof}.

\begin{theorem}[\protect{\cite[Proposition 3.3]{FT90}}] \label{THfavtarProp33}
Let $f: \mathbb{Z}^{n} \to \mathbb{R} \cup \{ +\infty  \}$
be a function with an integrally convex effective domain.
Then the following properties are equivalent:

{\rm (a)}
$f$ is integrally convex.

{\rm (b)}
For every $x, y \in \dom f$ with $\| x - y \|_{\infty} =2$ 
we have \ 
\begin{equation}  \label{intcnvconddist2}
\tilde{f}\, \bigg(\frac{x + y}{2} \bigg) 
\leq \frac{1}{2} (f(x) + f(y)).
\end{equation}
\vspace{-1.7\baselineskip}
\\
\finbox
\end{theorem}

\begin{theorem}[\protect{\cite[Proposition 3.1]{FT90}}; 
  see also \protect{\cite[Theorem 3.21]{Mdcasiam}}] 
  \label{THintcnvlocopt}
Let $f: \mathbb{Z}^{n} \to \mathbb{R} \cup \{ +\infty  \}$
be an integrally convex function and $x^{*} \in \dom f$.
Then $x^{*}$ is a minimizer of $f$
if and only if
$f(x^{*}) \leq f(x^{*} +  d)$ for all 
$d \in  \{ -1, 0, +1 \}^{n}$.
\finbox
\end{theorem}

The local characterization of global minima
stated in Theorem~\ref{THintcnvlocopt} above
can be generalized to the following form;
see Fig.~\ref{FGboxbarrier}.

\begin{figure}\begin{center}
\includegraphics[height=40mm]{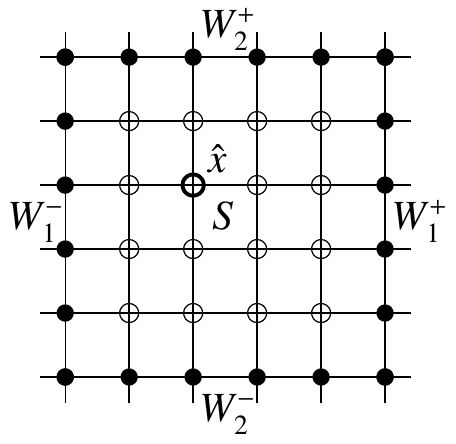}
\caption{Box-barrier property ($\circ \in S$, $\bullet \in W$)}
\label{FGboxbarrier}
\end{center}\end{figure}

\begin{theorem}[Box-barrier property] \label{THintcnvbox}
Let $f: \mathbb{Z}^{n} \to \mathbb{R} \cup \{ +\infty  \}$
be an integrally convex function, 
and let
$p \in (\mathbb{Z} \cup \{ -\infty  \})^{n}$
and
$q \in (\mathbb{Z} \cup \{ +\infty  \})^{n}$,
where $p \leq q$.
Define 
\begin{align*}
S &= \{ x \in \mathbb{Z}^{n} \mid p_{i} < x_{i} < q_{i} \ (i=1,\ldots,n) \},
\\
W_{i}^{+} &= \{ x \in \mathbb{Z}^{n} \mid 
 x_{i} = q_{i}, \  p_{j} \leq x_{j} \leq q_{j} \ (j \not= i) \}
\quad (i=1,\ldots,n),
\\
W_{i}^{-} &= \{ x \in \mathbb{Z}^{n} \mid 
 x_{i} = p_{i}, \  p_{j} \leq x_{j} \leq q_{j} \ (j \not= i) \}
\quad (i=1,\ldots,n),
\end{align*}
and $W = \bigcup_{i=1}^{n} (W_{i}^{+} \cup W_{i}^{-})$.
Let $\hat x \in S \cap \dom f$.
If 
$f(\hat x) \leq f(y)$ for all $y \in W$,
then 
$f(\hat x) \leq f(z)$ for all $z \in \ZZ^{n} \setminus S$.
\end{theorem}
\begin{proof}
Let
$U = \bigcup_{i=1}^{n}  \{ x \in \mathbb{R}^{n} \mid 
 x_{i} \in  \{ p_{i}, q_{i} \}, 
 \  p_{j} \leq x_{j} \leq q_{j} \ (j \not= i) \}$,
for which we have $U \cap \mathbb{Z}^{n} = W$. 
For a point $z \in \ZZ^{n} \setminus S$, 
the line segment connecting $\hat x$ and $z$ 
intersects $U$
at a point, say, $u \in \RR^{n}$.
Then its integral neighborhood $N(u)$ 
is contained in $W$.
Since the local convex extension $\tilde{f}(u)$ 
is a convex combination of 
the $f(y)$'s  with $y \in N(u)$,
and 
$f(y) \geq f(\hat x)$ for every $y \in W$,
we have 
$\tilde{f}(u) \geq f(\hat x)$.
On the other hand, it follows from integral convexity that 
$\tilde{f}(u) \leq (1 - \lambda) f(\hat x) + \lambda f(z)$
for some $\lambda$ with $0 < \lambda \leq 1$.
Hence 
$f(\hat x) \leq \tilde{f}(u) \leq (1 - \lambda) f(\hat x) + \lambda f(z)$,
and therefore,
$f(\hat x) \leq f(z)$.
\end{proof}

Theorem~\ref{THintcnvlocopt} is
a special case of Theorem~\ref{THintcnvbox}
with $p=\hat x - \bm{1}$ and $q=\hat x + \bm{1}$.
Another special case of Theorem~\ref{THintcnvbox}
with 
$p_{j} = -\infty$ $(j=1,\ldots,n)$ and 
$q_{j} = +\infty$ $(j \not= i)$ for a particular $i$
takes the following form,
which we use in Section~\ref{SCproxproofn2}.

\begin{corollary}[Hyperplane-barrier property] \label{COintcnvcutal}
Let $f: \mathbb{Z}^{n} \to \mathbb{R} \cup \{ +\infty  \}$
be an integrally convex function.
Let 
$\hat x \in \dom f$,
$q \in \ZZ$, and let 
$i$ be an integer with $1 \leq i \leq n$.
If 
$\hat x_{i} < q$
and
$f(\hat x) \leq f(y)$ for all $y \in \ZZ^{n}$ with $y_{i} = q$,
then 
$f(\hat x) \leq f(z)$ for all $z \in \ZZ^{n}$ with $z_{i} \geq q$.
\finbox
\end{corollary}

We denote the sets of  nonnegative integers and positive integers
 by $\ZZ_{+}$ and $\ZZ_{++}$,  respectively.
For $\alpha \in \ZZ$ we write $\alpha \ZZ$ for $\{ \alpha x \mid x \in \ZZ \}$.
For vectors $a, b \in \RR\sp{n}$ with $a \leq b$,
$[a,b]_{\RR}$ denotes the interval between $a$ and $b$,
i.e.,
$[a,b]_{\RR} = \{ x \in \RR\sp{n} \mid a \leq x \leq b \}$,
and 
$[a,b]_{\ZZ}$ the integer interval between $a$ and $b$,
i.e.,
$[a,b]_{\ZZ} = \{ x \in \ZZ\sp{n} \mid a \leq x \leq b \}$.

\section{The Scaling Operation for Integrally Convex Functions}
\label{SCscalingIC}

In this section we consider the scaling operation for integrally convex functions.
Recall that, for
$f: \mathbb{Z}^{n} \to \mathbb{R} \cup \{ +\infty  \}$
and $\alpha \in \mathbb{Z}_{++}$,
the $\alpha$-scaling of $f$ is defined to be the 
function $f^{\alpha}: \mathbb{Z}^{n} \to \mathbb{R} \cup \{ +\infty  \}$
given by
$f^{\alpha}(x) = f(\alpha x) $ $(x \in \mathbb{Z}^{n})$.

When $n = 2$, integral convexity is preserved under scaling.
We first deal with integrally convex sets.

\begin{proposition}  \label{PRintcnvsetscdim2}
Let $S \subseteq \mathbb{Z}^{2}$ be an integrally convex set
and $\alpha \in \mathbb{Z}_{++}$.
Then 
$S^{\alpha} = \{ x \in \mathbb{Z}^{2} \mid \alpha x \in S \}$
is an integrally convex set.
\end{proposition}
\begin{proof}
By Proposition~\ref{PRintcnvsetdim2} we can assume that $S$ is represented as 
$S = \{ (x_1,x_2) \in \ZZ\sp{2} \mid p_{i} x_{1} +  q_{i} x_{2} \leq r_{i}
\ (i=1,\ldots,m) \}$
for some $p_{i}, q_{i} \in \{ -1,0,+1 \}$ and $r_{i} \in \ZZ$ $(i=1,\ldots,m)$.
Since 
$(y_1, y_2) \in S^{\alpha}$ if and only if
$(\alpha y_1, \alpha y_2) \in S$,
we have
\begin{align*}
S^{\alpha} 
&= \{ (y_1,y_2) \in \ZZ\sp{2} \mid \alpha (p_{i} y_{1} +  q_{i} y_{2}) \leq r_{i}
\ (i=1,\ldots,m) \}
\\ &= 
\{ (y_1,y_2) \in \ZZ\sp{2} \mid p_{i} y_{1} +  q_{i} y_{2} \leq r'_{i}
\ (i=1,\ldots,m) \},
\end{align*}
where
$r'_{i} = \lfloor r_{i} / \alpha \rfloor$  
$(i=1,\ldots,m)$.
By Proposition~\ref{PRintcnvsetdim2}
this implies integral convexity of $S^{\alpha}$. 
\end{proof}

Next we turn to integrally convex functions.

\begin{theorem} \label{THintcnvscdim2}
Let $f: \mathbb{Z}^{2} \to \mathbb{R} \cup \{ +\infty  \}$
 be an integrally convex function and $\alpha \in \mathbb{Z}_{++}$.
Then the scaled function $f^{\alpha}$ is integrally convex.
\end{theorem}

\begin{proof}
The effective domain
$\dom f^{\alpha} = ( \dom f \cap (\alpha \ZZ)^{2} )/\alpha$
is an integrally convex set 
by Proposition~\ref{PRintcnvsetscdim2}.
By Theorem~\ref{THfavtarProp33} and Proposition~\ref{PRintcnvinvar},
we only have to check condition (\ref{intcnvconddist2})
for $f^{\alpha}$ with 
$x =(0,0)$ and $y = (2,0)$, $(2, 2)$, $(2, 1)$.
That is, it suffices to show
\begin{align}
& f(0,0) + f(2 \alpha, 0)  \geq  2 f(\alpha,0) , 
\label{prfdim2scal20}
\\
& f(0,0) + f(2 \alpha, 2 \alpha)  \geq  2 f(\alpha,\alpha), 
\label{prfdim2scal22}
\\
& f(0,0) + f(2 \alpha, \alpha)  \geq  f(\alpha,\alpha) + f(\alpha,0). 
\label{prfdim2scal21}
\end{align}
The first two inequalities, (\ref{prfdim2scal20}) and (\ref{prfdim2scal22}),
follow easily from integral convexity of $f$,
whereas (\ref{prfdim2scal21}) is a special case of
the basic parallelogram inequality (\ref{paraineq2dim}) below
with $a = b = \alpha$.
\end{proof}

\begin{proposition} [Basic parallelogram inequality] \label{PRparallelineqdim2}
For an integrally convex function 
$f: \mathbb{Z}^{2} \to \mathbb{R} \cup \{ +\infty  \}$
we have
\begin{equation} \label{paraineq2dim}
 f(0,0) + f(a+b,a) \geq  f(a,a) + f(b,0) 
\qquad (a,b \in \ZZ_{+}). 
\end{equation}
\end{proposition}
\begin{proof}
We may assume $a, b \geq 1$ and
$\{ (0,0), (a+b,a) \} \subseteq \dom f$,
since otherwise the inequality (\ref{paraineq2dim}) is trivially true.
Since $\dom f$ is integrally convex, Corollary \ref{COintconvset-pairofpoints} implies that
$k(1,1) +l(1,0) \in \dom f$ for all $(k,l)$ with $0 \leq k \leq a$ and $0 \leq l\leq b$.
We use the notation $f_{x}(z)=f(x+z)$.
For each $x \in \dom f$  we have
\[ 
 f_{x}( 0,0 ) +  f_{x}( 2,1 ) \geq f_{x}( 1,1 ) +  f_{x}( 1,0 )
\] 
by integral convexity of $f$.
By adding these inequalities for $x=k(1,1) +l(1,0)$ 
with $0 \leq k \leq a-1$ and $0 \leq l \leq b-1$,
we obtain (\ref{paraineq2dim}).
Note that all the terms involved in these inequalities are finite,
since $k(1,1) +l(1,0) \in \dom f$ for all $k$ and $l$.
\end{proof}

If $n \geq 3$, $f^{\alpha}$  is not always integrally convex.
This is demonstrated by the following example.

\begin{example} \rm \label{EXscalingNG422}
Consider the integrally convex function
$f: \mathbb{Z}^{3} \to \mathbb{R} \cup \{ +\infty  \}$
defined on $\dom f = [(0,0,0), (4,2,2)]_{\ZZ}$ by
\[
\begin{array}{c|rrrrrl}
  \multicolumn{1}{c}{x_{2}}  &  \multicolumn{5}{c}{f(x_{1},x_{2},0)}  
\\ \cline{2-6}
 2  & 3 & 1 & 1 & 1 & \multicolumn{1}{r|}{3}
\\
 1  & 1 & 0 & 0 & 0 & \multicolumn{1}{r|}{0} 
\\
 0  & 0 & 0 & 0 & 0 & \multicolumn{1}{r|}{3}
\\ \cline{1-6}
  & 0 & 1 & 2 & 3 & 4 & x_{1} 
\end{array}  
\quad
\begin{array}{c|rrrrrl}
  \multicolumn{1}{c}{x_{2}}  &  \multicolumn{5}{c}{f(x_{1},x_{2},1)}  
\\ \cline{2-6}
 2  & 2 & 1 & 0 & 0 & \multicolumn{1}{r|}{0}
\\
 1  & 1 & 0 & 0 & 0 & \multicolumn{1}{r|}{0} 
\\
 0  & 0 & 0 & 0 & 0 & \multicolumn{1}{r|}{0}
\\ \cline{1-6}
  & 0 & 1 & 2 & 3 & 4 & x_{1} 
\end{array}  
\quad
\begin{array}{c|rrrrrl}
  \multicolumn{1}{c}{x_{2}}  &  \multicolumn{5}{c}{f(x_{1},x_{2},2)}  
\\ \cline{2-6}
 2  & 3 & 2 & 1 & 0 & \multicolumn{1}{r|}{0}
\\
 1  & 2 & 1 & 0 & 0 & \multicolumn{1}{r|}{0} 
\\
 0  & 3 & 0 & 0 & 0 & \multicolumn{1}{r|}{3}
\\ \cline{1-6}
  & 0 & 1 & 2 & 3 & 4 & x_{1} 
\end{array}  
\]
For the scaling with $\alpha = 2$, we have a failure of
integral convexity.
Indeed, for $x = (0,0,0)$ and $y = (2, 1, 1)$ we have
\begin{align*}
\widetilde {f^{\alpha}} \, \bigg(\frac{x + y}{2} \bigg)
&= \min \{
\frac{1}{2}f^{\alpha}(1,1,1) + \frac{1}{2}f^{\alpha}(1,0,0), \ 
\frac{1}{2}f^{\alpha}(1,1,0) + \frac{1}{2}f^{\alpha}(1,0,1) \}
\\
 & = \frac{1}{2}  \min \{ f(2,2,2) + f(2,0,0), \  f(2,2,0) + f(2,0,2) \}  
\\
 & = \frac{1}{2}  \min \{ 1+ 0, \   1 + 0  \} = \frac{1}{2}
\\
&  > 0 = \frac{1}{2} ( f(0,0,0) + f(4,2,2) ) 
  = \frac{1}{2} (f^{\alpha}(x) + f^{\alpha}(y)),
\end{align*}
which shows the failure of (\ref{intcnvconddist2}) in Theorem \ref{THfavtarProp33}.
The set
$S = \argmin f = \{ x \mid f(x)=0 \}$
is an integrally convex set, and 
$S^{\alpha} = \{ x \mid \alpha x \in S \}
= \{ (0,0,0), (1,0,0), (1,0,1), (2,1,1) \}$
is not an integrally convex set.
\finbox
\end{example}

In view of the fact that the class of ${\rm L}^{\natural}$-convex functions
is stable under scaling, while this is not true for the superclass
of integrally convex functions,
we are naturally led to the question of 
finding an intermediate class of functions that is stable under scaling.
See Section~\ref{SCconrem} for this issue.

\section{Preliminary Discussion on Proximity Theorems}
\label{SCpredisproxIC}

Let $f: \mathbb{Z}^{n} \to \mathbb{R} \cup \{ +\infty  \}$ 
and $\alpha \in \ZZ_{++}$.
We say that $x^{\alpha} \in \dom f$ is an {\em $\alpha$-local minimizer} of $f$
(or {\em $\alpha$-local minimal} for $f$)
if $f(x^{\alpha}) \leq f(x^{\alpha}+ \alpha d)$
for all $d \in  \{ -1,0, +1 \}^{n}$.
In general terms a proximity theorem states that
for $\alpha \in \ZZ_{++}$ there exists an integer
$B(n,\alpha) \in \ZZ_{+}$ such that
if $x^{\alpha}$ is an $\alpha$-local minimizer of $f$,
then there exists a minimizer $x^{*}$ of $f$ satisfying 
$ \| x^{\alpha} - x^{*}\|_{\infty} \leq B(n,\alpha)$,
where $B(n,\alpha)$ is called the {\em proximity distance}.

Before presenting a proximity theorem
for integrally convex functions in Section~\ref{SCproxthmICgeneral},
we establish in this section 
lower bounds for the proximity distance.
We also present a proximity theorem for $n=2$,
as the proof is fairly simple in this particular case, 
though the proof method does not extend to general $n \geq 3$.

\subsection{Lower bounds for the proximity distance}

The following examples provide us with lower bounds for the proximity distance.
The first three demonstrate the tightness of the bounds 
for separable convex functions,
${\rm L}^{\natural}$-convex and ${\rm M}^{\natural}$-convex functions given in 
Theorems \ref{THsepfnproximity}, \ref{THlfnproximity} and \ref{THmfnproximity}, respectively.

\begin{example}[Separable convex function] \rm \label{EXsepfnproximitytight}
Let $\varphi(t) = \max(-t, (\alpha -1)(t - \alpha) )$ for $t \in \ZZ$
and define 
$f(x) = \varphi(x_{1}) + \cdots + \varphi(x_{n})$,
which is separable convex.
This function has a unique minimizer at $x^{*}=(\alpha -1, \ldots, \alpha -1)$,
whereas $x^{\alpha} = \veczero$ is  $\alpha$-local minimal
and $ \| x^{\alpha} - x^{*}\|_{\infty} = \alpha - 1$.
This shows the tightness of the bound $\alpha - 1$ given in Theorem~\ref{THsepfnproximity}.
\finbox
\end{example}

\begin{example}[${\rm L}^{\natural}$-convex function] \rm \label{EXproxtigntLnat}
Consider $X \subseteq \ZZ^{n}$ defined by
\begin{align*}
X &= \{ x \in \ZZ^{n} \mid 
0 \leq x_{i} - x_{i+1} \leq \alpha -1 \ (i=1,\ldots,n-1), \ 
0 \leq x_{n} \leq \alpha -1  \}
\\
 &= \{ x \in \ZZ^{n} \mid 
   x = \sum_{i=1}^{n} \mu_{i} \unitvec{ \{ 1,2,\ldots, i \} },  \ \ 
   0 \leq \mu_{i} \leq \alpha -1 \ (i=1,\ldots,n)   
 \},
\end{align*}
where
$ \unitvec{ \{ 1,2,\ldots, i \} } 
= ( \overbrace{1,1,\ldots, 1}^{i}, 0,0,\ldots, 0)$. 
The function $f$ defined by $f(x)=-x_{1}$ on $\dom f =X$ is
an ${\rm L}^{\natural}$-convex function
and  has a unique minimizer at 
$x^{*} =  (n (\alpha -1),(n-1)(\alpha -1),\ldots,2(\alpha -1),\alpha -1)$.
On the other hand, $x^{\alpha} = \veczero$ is  $\alpha$-local minimal,
since
$X \cap \{ -\alpha, 0,\alpha \}^{n} = \{ \bm{0} \}$.
We have $ \| x^{\alpha} - x^{*}\|_{\infty} = n (\alpha - 1)$,
which shows the tightness of the bound $n (\alpha - 1)$ given in Theorem~\ref{THlfnproximity}.
This example is a reformulation of \cite[Remark~2.3]{MT02proxRIMS} 
for {\rm L}-convex functions to ${\rm L}^{\natural}$-convex functions.
\finbox
\end{example}

\begin{example}[${\rm M}^{\natural}$-convex function] \rm \label{EXproxtigntMnat}
Consider $X \subseteq \ZZ^{n}$ defined by
\begin{align*}
X &= \{ x \in \ZZ^{n} \mid 
  0 \leq  x_{1} + x_{2} + \cdots + x_{n} \leq \alpha -1 , \  
  -(\alpha -1) \leq x_{i} \leq 0 \ (i=2,\ldots,n) \}
\\
 &= \{ x \in \ZZ^{n} \mid 
 x = (\mu_{1} + \mu_{2} + \cdots + \mu_{n}, -\mu_{2},-\mu_{3}, \ldots, -\mu_{n}), \ \ 
   0 \leq \mu_{i} \leq \alpha -1 \ (i=1,\ldots,n)   
 \} .
\end{align*}
The function $f$ defined by $f(x)=-x_{1}$ on $\dom f =X$ is
an ${\rm M}^{\natural}$-convex function
and has a unique minimizer at 
$x^{*} = (n (\alpha -1),-(\alpha -1), -(\alpha -1),\ldots,-(\alpha -1))$.
On the other hand, $x^{\alpha} = \veczero$ is  $\alpha$-local minimal,
since
$X \cap \{ -\alpha, 0,\alpha \}^{n} = \{ \bm{0} \}$.
We have $ \| x^{\alpha} - x^{*}\|_{\infty} = n (\alpha - 1)$,
which shows the tightness of the bound $n (\alpha - 1)$ given in Theorem~\ref{THmfnproximity}.
This example is a reformulation of \cite[Remark~2.8]{MT02proxRIMS} 
for {\rm M}-convex functions to ${\rm M}^{\natural}$-convex functions.
\finbox
\end{example}

For integrally convex functions with $n \geq 3$, 
the bound $n (\alpha -1)$ is no longer valid.
This is demonstrated by the following examples.

\begin{example}  \rm \label{EXproxNG422}
Consider an integrally convex function
$f: \mathbb{Z}^{3} \to \mathbb{R} \cup \{ +\infty  \}$
defined on
$\dom f = [(0,0,0), (4,2,2)]_{\ZZ}$ by
\addtolength{\arraycolsep}{-1pt}%
\[
\begin{array}{c|rrrrrl}
  \multicolumn{1}{c}{x_{2}}  &  \multicolumn{5}{c}{f(x_{1},x_{2},0)}  
\\ \cline{2-6}
 2  & 5 & 1 & 0 & 0 & \multicolumn{1}{r|}{4}
\\
 1  & 2 &-1 &-2 & 0 & \multicolumn{1}{r|}{3} 
\\
 0  & 0 &-1 & 0 & 1 & \multicolumn{1}{r|}{6}
\\ \cline{1-6}
  & 0 & 1 & 2 & 3 & 4 & x_{1} 
\end{array}  
\quad
\begin{array}{c|rrrrrl}
  \multicolumn{1}{c}{x_{2}}  &  \multicolumn{5}{c}{f(x_{1},x_{2},1)}  
\\ \cline{2-6}
 2  & 4 & 1 &-2 &-3 & \multicolumn{1}{r|}{-1}
\\
 1  & 2 &-1 &-2 &-3 & \multicolumn{1}{r|}{-1} 
\\
 0  & 2 &-1 &-2 & 0 & \multicolumn{1}{r|}{5}
\\ \cline{1-6}
  & 0 & 1 & 2 & 3 & 4 & x_{1} 
\end{array}  
\quad
\begin{array}{c|rrrrrl}
  \multicolumn{1}{c}{x_{2}}  &  \multicolumn{5}{c}{f(x_{1},x_{2},2)}  
\\ \cline{2-6}
 2  & 6 & 3 & 0 &-3 & \multicolumn{1}{r|}{-4}
\\
 1  & 6 & 1 &-2 &-3 & \multicolumn{1}{r|}{1} 
\\
 0  & 6 & 2 & 0 & 3 & \multicolumn{1}{r|}{6}
\\ \cline{1-6}
  & 0 & 1 & 2 & 3 & 4 & x_{1} 
\end{array}  
\]
\addtolength{\arraycolsep}{1pt}%
and let $\alpha = 2$.
For $x^{\alpha}=(0,0,0)$ we have
$f(x^{\alpha})=0$ and
$f(x^{\alpha}) \leq f(x^{\alpha}+ 2 d)$
for $d=(1,0,0), (0,1,0)$, $(0,0,1)$, $(1,1,0), (1,0,1), (0,1,1), (1,1,1)$.
Hence $x^{\alpha}=(0,0,0)$ is  $\alpha$-local minimal.
A unique (global) minimizer of $f$ is located at $x^{*}=(4,2,2)$ with
$f(x^{*})=-4$ and $\| x^{\alpha} - x^{*}  \|_{\infty} = 4$.
The $\ell_{\infty}$-distance between 
$x^{\alpha}$ and $x^{*}$ is strictly larger than $n (\alpha -1) = 3$.
We remark that the scaled function $f^{\alpha}$
is not integrally convex.
\finbox
\end{example}

The following example demonstrates a quadratic lower bound in $n$
for the proximity distance for integrally convex functions.

\begin{example} \rm \label{EXproxtigntM2}
For a positive integer $m \geq 1$, we consider two bipartite graphs 
$G_{1}$ and $G_{2}$ on vertex bipartition
$(\{ 0^{+}, 1^{+},\ldots,m^{+} \}, \{ 0^{-}, 1^{-},\ldots,m^{-} \})$;
see Fig.~\ref{FGm2proxtight}.
The edge sets of $G_{1}$  and $G_{2}$ are defined respectively as
$E_{1} = \{ (0^{+}, 0^{-}) \} \cup \{ (i^{+}, j^{-}) \mid i,j=1,\ldots,m \}$
and
$E_{2} = \{ (0^{+}, j^{-}) \mid j=1,\ldots,m \} 
   \cup \{ (i^{+}, 0^{-}) \mid i=1,\ldots,m \}$.
Let 
$V^{+} = \{ 1^{+},\ldots,m^{+} \}$, 
$V^{-} = \{ 1^{-},\ldots,m^{-} \}$,
and $n = 2m + 2$.
Consider $X_{1}, X_{2} \subseteq \ZZ^{n}$ defined by
\begin{align*}
  X_{1} &= \left\{ \sum_{i=1}^{m}\sum_{j=1}^{m}
      \lambda_{ij}(\unitvec{i^{+}}{-}\unitvec{j^{-}}) +
      \lambda_{0} (\unitvec{0^{+}}{-}\unitvec{0^{-}})\;
      \begin{array}{|l}
       \lambda_{ij} \in [0,\alpha -1]_{\ZZ} \; (i,j=1,\ldots,m) \\
       \lambda_{0} \in [0,m^{2}(\alpha -1)]_{\ZZ}
      \end{array} \right\},
\\
  X_{2} &= \left\{ \sum_{i=1}^{m}\mu_{i}(\unitvec{i^{+}}{-}\unitvec{0^{-}}) +
       \sum_{j=1}^{m}\nu_{j}(\unitvec{0^{+}}{-}\unitvec{j^{-}})\;
      \begin{array}{|l}
       \mu_{i} \in [0,m(\alpha -1)]_{\ZZ} \; (i=1,\ldots,m) \\
       \nu_{j} \in [0,m(\alpha -1)]_{\ZZ} \; (j=1,\ldots,m)
      \end{array} \right\},
\end{align*}
where $X_{1}$ and $X_{2}$ represent the sets of boundaries of flows in $G_{1}$ and $G_{2}$, respectively.
We define functions 
$f_{1},f_{2} : \ZZ^{n} \to \RR \cup \{ +\infty \}$ 
with 
$\dom f_{1} = X_{1}$ and $\dom f_{2} = X_{2}$
by
\[
   f_{1}(x) = \left\{ \begin{array}{ll}
    x(V^{-})  & (x \in X_{1}) , \\
    +\infty     & (x \not\in X_{1}),
   \end{array}\right.
   \quad
   f_{2}(x) = \left\{ \begin{array}{ll}
    x(V^{-})  & (x \in X_{2}) , \\
    +\infty     & (x \not\in X_{2})
   \end{array}\right.
   \qquad (x \in \ZZ^{n}),
\]
where $x(U) = \sum_{u \in U} x_{u}$ for any set $U$ of vertices. 
Both $f_{1}$ and $f_{2}$ are {\rm M}-convex, 
and hence $f = f_{1} + f_{2}$ is an ${\rm M}_{2}$-convex function, 
which is integrally convex (see \cite[Section~8.3.1]{Mdcasiam}).
We have
$\dom f = \dom f_{1} \cap \dom f_{2} = X_{1} \cap X_{2}$
and $f$ is linear on $\dom f$.
As is easily verified, $f$ has a unique minimizer at $x^{*}$ defined by
\[
 x^{*}_{u} = \left\{\begin{array}{rl}
             m(\alpha-1) & (u \in V^{+}), \\
            -m(\alpha-1) & (u \in V^{-}), \\
       m^{2}(\alpha-1) & (u = 0^{+}), \\
      -m^{2}(\alpha-1) & (u = 0^{-}),
    \end{array}\right.
\]
which corresponds to 
$\lambda_{0} = m^{2}(\alpha -1)$, \ 
$\lambda_{ij} = \alpha -1$, \ 
$\mu_{i} = \nu_{j} = m(\alpha -1)$ \ $(i,j=1,\ldots,m)$.
We mention that the function $f$ here is constructed in \cite[Remark~2.19]{MT02proxRIMS}
for a slightly different purpose (i.e., for ${\rm M}_{2}$-proximity theorem).

\begin{figure}\begin{center}
\includegraphics[height=35mm]{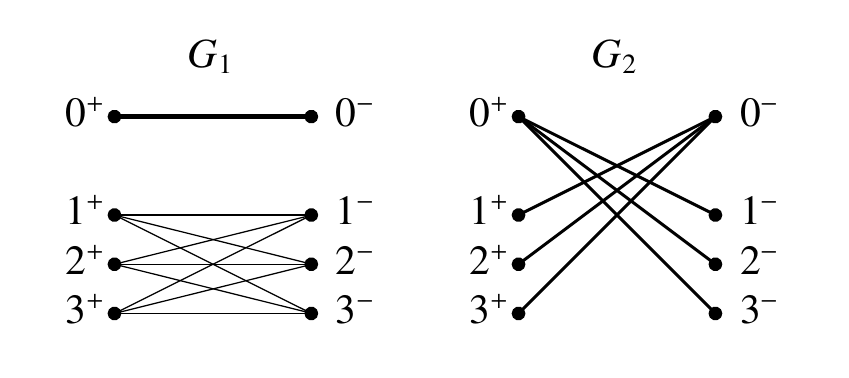}
\caption{Example for $O(n\sp{2})$ lower bound for proximity distance ($m=3$).}
\label{FGm2proxtight}
\end{center}\end{figure}

Let $x^{\alpha} = \veczero$.
Obviously, $\veczero \in \dom f$. Moreover, 
$x^{\alpha} = \veczero$ is  $\alpha$-local minimal,
since 
$\dom f \cap \{ -\alpha, 0,\alpha \}^{n} = \{ \bm{0} \}$,
as shown below.
Since
$\| x^{*}-x^{\alpha}\|_{\infty} = m^{2}(\alpha -1)
= (n-2)^{2}(\alpha -1)/4$, 
we obtain a quadratic lower bound 
$(n-2)^{2}(\alpha -1)/4$
for the proximity distance for integrally convex functions.

The proof of 
$\dom f \cap \{ -\alpha, 0,\alpha \}^{n} = \{ \bm{0} \}$
 goes as follows.
Let $x \in X_{1} \cap X_{2} \cap \{ -\alpha, 0,\alpha \}^{n}$.
We have $x_{0^{+}} \in \{ 0, \alpha \}$ and $x_{0^{-}} \in \{ 0, -\alpha \}$.
We consider four cases to conclude that $x = \bm{0}$.

(i) Case of $x_{0^{+}} = x_{0^{-}} = 0$: 
The structure of $X_{2}$ forces $x = \bm{0}$.

(ii) Case of  $x_{0^{+}} = \alpha$, $x_{0^{-}} = 0$: 
The structure of $X_{2}$ forces
$x_{i^{+}} = 0$ for $i=1, \ldots, m$ and
\[
x_{j^{-}} = 
 \left\{\begin{array}{rl}
            -\alpha & (j = j_{0}) , \\
             0 &   (j \not= j_{0}) \\
    \end{array}\right.
\]
for some $j_{0}$ $(1 \leq j_{0} \leq m$), 
but this is impossible by the structure of $X_{1}$.

(iii) Case of $x_{0^{+}} =  0$, $x_{0^{-}} = -\alpha$: 
The proof is similar to that of (ii) above.

(iv)  Case of  $x_{0^{+}} = \alpha$, $x_{0^{-}} = -\alpha$:
The structure of $X_{2}$ forces 
\[
x_{i^{+}} = 
 \left\{\begin{array}{rl}
             \alpha & (i = i_{0}) , \\
             0 &   (i \not= i_{0}) , \\
    \end{array}\right.
\quad
x_{j^{-}} = 
 \left\{\begin{array}{rl}
            -\alpha & (j = j_{0}), \\
             0 &   (j \not= j_{0}) \\
    \end{array}\right.
\]
for some $i_{0}$ $(1 \leq i_{0} \leq m$) and $j_{0}$ $(1 \leq j_{0} \leq m$),
but this is impossible by the structure of $X_{1}$.
\finbox
\end{example}

We have seen that the proximity theorem with the linear bound $n (\alpha -1)$
does not hold for all integrally convex functions.
Then a natural question arises:
can we establish a proximity theorem at all
by enlarging the proximity bound?
This question is answered in the affirmative 
in Section~\ref{SCproxthmICgeneral}.

\subsection{A proximity theorem for integrally convex functions with $n=2$}

In the case of $n=2$ the proximity bound
$n (\alpha - 1) = 2 (\alpha - 1)$ is valid
for integrally convex functions%
\footnote{
Recall that $n=3$ in Example \ref{EXproxNG422}.
}. 

\begin{theorem}  \label{THproxintcnvdim2alpha}
Let $f: \mathbb{Z}^{2} \to \mathbb{R} \cup \{ +\infty  \}$
be an integrally convex function,
$\alpha \in \mathbb{Z}_{++}$, 
and $x^{\alpha} \in \dom f$.
If $f(x^{\alpha}) \leq f(x^{\alpha}+ \alpha d)$ for all 
$d \in  \{ -1,0, +1 \}^{2}$,
then there exists a minimizer 
$x^{*} \in \ZZ^{2}$ of $f$ with
$  \| x^{\alpha} - x^{*}  \|_{\infty} \leq 2 (\alpha - 1)$.
\end{theorem}

\begin{proof}
We may assume $\alpha \geq 2$ and $x^{\alpha} = \bm{0}$ 
by Proposition \ref{PRintcnvinvar} (1).
Define 
\begin{align*}
C &= \{ (x_{1},x_{2}) \in \ZZ^{2} \mid 0 \leq x_{2} \leq x_{1} \},
\\
S &= \{ (x_{1},x_{2}) \in \ZZ^{2} \mid 0 \leq x_{2} \leq x_{1} \leq 2 (\alpha -1) \}.
\end{align*}
Let $\mu$ be the minimum of $f(x_{1},x_{2})$ over $(x_{1},x_{2}) \in S$
and let $(\hat x_{1},\hat x_{2})$ be a point in $S$ with 
$f(\hat x_{1},\hat x_{2})=\mu$.
Then 
\begin{equation} \label{prfT0al}
f(x_{1},x_{2}) \geq \mu 
\qquad
((x_{1},x_{2}) \in S).
\end{equation}
We will show that
\begin{equation} \label{prfT1al}
  f(2 \alpha - 1,k) \geq \mu 
\qquad (0 \leq k \leq 2 \alpha -1) .
\end{equation}
Then, by Corollary~\ref{COintcnvcutal} 
(hyperplane-barrier property), it follows 
that 
$f(z_{1},z_{2}) \geq \mu$ for all $(z_{1},z_{2}) \in C$, 
that is, 
there is no $(z_{1},z_{2}) \in C \setminus S$ with 
$f(z_{1},z_{2}) <  \mu$.
This proves the claim of the theorem,
since $\ZZ^{2}$ can be covered by eight sectors similar to $C$
and Proposition~\ref{PRintcnvinvar} holds.

The basic parallelogram inequality (\ref{paraineq2dim})
with $a = k$ and $b =2 \alpha -1-k$ yields
\begin{equation} \label{prfT2al}
 f(0,0) + f(2 \alpha - 1,k) \geq  f(k,k) + f(2 \alpha -1-k,0) .
\end{equation}

\paragraph{Case 1: $0 \leq k \leq \alpha -1$.}
Since $2 \alpha -1-k \geq \alpha$, 
by convexity of $f(t,0)$ in $t$, we have
\[
\frac{1}{2 \alpha -1-k}[f(2 \alpha -1-k,0) - f(0,0)] 
\geq \frac{1}{\alpha}[f(\alpha,0) - f(0,0)] \geq 0.
\]
On the other hand, $f(k,k) \geq \mu$ by (\ref{prfT0al}).
Then it follows from (\ref{prfT2al}) that
\[
 f(2 \alpha - 1,k) \geq  f(k,k) +[ f(2 \alpha -1-k,0) - f(0,0) ]  \geq  \mu.
\]

\paragraph{Case 2: $\alpha \leq k \leq 2 \alpha -1$.}
Since $k \geq \alpha$, 
by convexity of $f(t,t)$ in $t$, we have
\[
\frac{1}{k}[f(k,k) - f(0,0)] 
\geq \frac{1}{\alpha}[f(\alpha,\alpha) - f(0,0)] \geq 0.
\]
On the other hand, $f(2 \alpha -1-k,0) \geq \mu$ by (\ref{prfT0al}).
Then it follows from (\ref{prfT2al}) that
\[
 f(2 \alpha - 1,k) \geq  f(2 \alpha -1-k,0) + [ f(k,k) - f(0,0) ]  \geq  \mu.
\]

We have thus shown (\ref{prfT1al}), completing the proof of 
Theorem~\ref{THproxintcnvdim2alpha}.
\end{proof}

\section{A Proximity Theorem for Integrally Convex Functions}
\label{SCproxthmICgeneral}

In this section we establish a proximity theorem 
for integrally convex functions in an arbitrary number of variables.

\subsection{Main result}

\begin{theorem}  \label{THproxintcnv}
Let $f: \mathbb{Z}^{n} \to \mathbb{R} \cup \{ +\infty  \}$
be an integrally convex function,
$\alpha \in \mathbb{Z}_{++}$, 
and $x^{\alpha} \in \dom f$.

\noindent
{\rm (1)}
If 
\begin{equation}\label{icproxlocopt0}
f(x^{\alpha}) \leq f(x^{\alpha}+ \alpha d)
\qquad  
(\forall \ d \in  \{ -1,0, +1 \}^{n}),
\end{equation}
then $\argmin f \not= \emptyset$ and
there exists $x^{*} \in \argmin f$ with
\begin{equation}\label{icproximity0}
 \| x^{\alpha} - x^{*}\|_{\infty} \leq \beta_{n} (\alpha - 1) ,
\end{equation}
where $\beta_{n}$ is defined by 
\begin{equation}\label{eqMrec1}
\beta_{1}=1, \quad \beta_{2}=2; \qquad  
 \beta_{n} =  \frac{n+1}{2} \beta_{n-1}  + 1
\quad (n=3,4,\ldots).
\end{equation}

\noindent
{\rm (2)}
The coefficient $\beta_{n}$ of the proximity bound satisfies
\begin{equation}\label{eqMrec2est}
 \beta_{n} \leq \frac{(n+1)!}{2^{n-1}} 
\qquad (n=3,4,\ldots).
\end{equation}
\finbox
\end{theorem}

The numerical values of $\beta_{n}$ and its bounds are as follows:
\begin{equation}\label{betatable}
\begin{array}{c|cccccc}
\hline
n &   2 & 3 & 4 & 5 & 6 & 7
\\ \hline
\mbox{Value \  by (\ref{eqMrec1})} 
  & \phantom{0}2\phantom{0} & \phantom{0}5\phantom{0}  & 13.5 & 41.5 & 146.25 & 586
\\
 \mbox{Bound by (\ref{eqMrec2est})} 
  & - & 6 & 15\phantom{.0} & 45\phantom{.0}  &  157.5\phantom{0}    & 630
\\ \hline
\end{array}
\end{equation}

\begin{remark} \rm 
The bound (\ref{icproximity0}) can be strengthened to
$ \| x^{\alpha} - x^{*}\|_{\infty} \leq \lfloor \beta_{n} (\alpha - 1) \rfloor $,
but 
$ \| x^{\alpha} - x^{*}\|_{\infty} \leq \lfloor \beta_{n} \rfloor (\alpha - 1)$
may not be correct (our proof does not justify this).
\finbox
\end{remark}

To prove Theorem~\ref{THproxintcnv} (1)
we first note that the theorem follows from its special case 
where $x^{\alpha} = \bm{0}$
and $f$ is defined on a bounded set in the nonnegative orthant $\ZZ_{+}^{n}$.
That is, the proof of Theorem~\ref{THproxintcnv} (1) is reduced to proving 
the following proposition.
We use the notation $N=\{1,2,\ldots,n\}$ and
$\unitvec{A}$ for the characteristic vector of $A \subseteq N$.

\begin{proposition}  \label{PRproxintcnvnonneg}
Let $\alpha \in \ZZ_{++}$ and
 $f: \mathbb{Z}^{n} \to \mathbb{R} \cup \{ +\infty  \}$
be an integrally convex function such that
$\dom f$ is a bounded subset of\, $\ZZ_{+}^{n}$ containing the origin $\bm{0}$.
If
\begin{equation}\label{assumption1}
 f(\bm{0}) \leq f(\alpha \unitvec{A})
 \qquad (\forall A \subseteq N),
\end{equation}
then 
there exists $x^{*} \in \argmin f$ with
\begin{equation}\label{icproximity}
 \| x^{*}\|_{\infty} \leq \beta_{n} (\alpha - 1) ,
\end{equation}
where $\beta_{n}$ is defined by {\rm (\ref{eqMrec1})}.
\finbox
\end{proposition}

Suppose that Proposition~\ref{PRproxintcnvnonneg} has been established.
Then Theorem~\ref{THproxintcnv} (1) can be derived 
from Proposition~\ref{PRproxintcnvnonneg} in three steps:

\begin{enumerate}
\item
We may assume $x^{\alpha} = \bm{0}$ 
by Proposition~\ref{PRintcnvinvar} (1).

\item
We may further assume that $\dom f$ is bounded.
Let $M$ be a sufficiently large integer,
say, $M \geq \beta_{n} (\alpha - 1) + 1$,
and  $f_{M}$ be the restriction of $f$ to the integer interval 
$[ -M \bm{1}, M \bm{1} ]_{\ZZ}$,
where $\bm{1}=(1,1,\ldots, 1)$.
Then $x^{\alpha} = \bm{0}$ is $\alpha$-local minimal for $f_{M}$. 
If the special case of Theorem~\ref{THproxintcnv} with  
$x^{\alpha} = \bm{0}$ and bounded $\dom f$ is true,  
then there exists $x^{*} \in \argmin f_{M}$ satisfying
$ \| x^{*}\|_{\infty} \leq \beta_{n} (\alpha - 1)$.
Since $x^{*} \in \argmin f_{M}$ we have
$f_{M}(x^{*}) \leq f_{M}(x^{*}+ d)$ \ 
$(\forall \, d \in  \{ -1,0, +1 \}^{n})$,
which implies
$f(x^{*}) \leq f(x^{*}+ d)$ \ $(\forall \, d \in  \{ -1,0, +1 \}^{n})$.
Then Theorem~\ref{THintcnvlocopt} shows that $x^{*} \in \argmin f$.

\item
We consider $2\sp{n}$ orthants separately. 
For each $s=(s_{1}, s_{2}, \ldots, s_{n}) \in \{ +1, -1 \}\sp{n}$
we consider the function 
$f_{s}(x)=f(s x)$ on $\ZZ_{+}^{n}$,
where $s x = (s_{1} x_{1}, s_{2} x_{2}, \ldots, s_{n} x_{n})$.
Noting that $\dom f_{s}$ is a bounded subset of  $\ZZ_{+}^{n}$,
we apply Proposition~\ref{PRproxintcnvnonneg} 
to $f_{s}$ to obtain $x^{*}_{s}$ with
$ \| x^{*}_{s} \|_{\infty} \leq \beta_{n} (\alpha - 1)$.
{}From among $2\sp{n}$ such $x^{*}_{s}$, take the one 
with the function value $f(s x^{*}_{s})$ minimum.
Then 
$x^{*} = s x^{*}_{s}$ is a minimizer of 
$f$, and satisfies 
$ \| x^{*} \|_{\infty} \leq \beta_{n} (\alpha - 1)$.
\end{enumerate}


\subsection{Tools for the proof: $f$-minimality}
\label{SCprooftools}

In this section we introduce some technical tools that we use in the proof of 
Proposition~\ref{PRproxintcnvnonneg}.

For $A\;(\neq \emptyset) \subseteq N$,
we consider a set of integer vectors  
\begin{align}
 B_{A} &=  
\{ \unitvec{A} + \unitvec{i}, \unitvec{A} - \unitvec{i} \mid i \in A \} 
          \cup \{\unitvec{A}+ \unitvec{i} \mid i \in N \setminus A \}
   \cup   \{\unitvec{A}\} ,
\label{generator}
\end{align}
and the cones of their nonnegative integer and real combinations
\begin{align}
 C_{A} &= \{ 
   \sum_{i \in A} \mu_{i}^{+}(\unitvec{A} + \unitvec{i}) + \sum_{i \in A} \mu_{i}^   {-} (\unitvec{A} - \unitvec{i})
   + \sum_{i \in N \setminus A} \mu_{i}^{\circ} (\unitvec{A}+ \unitvec{i})
  + \lambda \unitvec{A} 
 \mid 
 \mu_{i}^{+},\mu_{i}^{-}, \mu_{i}^{\circ}, \lambda \in \ZZ_{+}
\},
\label{coneCAdef}
\\
 \tilde{C}_{A} &= \{ 
   \sum_{i \in A} \mu_{i}^{+}(\unitvec{A} + \unitvec{i}) + \sum_{i \in A} \mu_{i}^   {-} (\unitvec{A} - \unitvec{i})
   + \sum_{i \in N \setminus A} \mu_{i}^{\circ} (\unitvec{A}+ \unitvec{i})
  + \lambda \unitvec{A} 
 \mid 
 \mu_{i}^{+},\mu_{i}^{-}, \mu_{i}^{\circ}, \lambda \in \RR_{+}
\},
\label{coneCAtdef}
\end{align}
where $C_{A}$ is often referred to as the integer cone generated by $B_{A}$.
We first note the following fact,
which provides us with a clearer geometric view,
though it is not used in the proof of
Proposition~\ref{PRproxintcnvnonneg}.

\begin{proposition}  \label{PRhilbertbase}
$B_{A}$ is a Hilbert basis of the convex cone $\tilde{C}_{A}$ generated by $B_{A}$.
That is,  $C_{A} = \tilde{C}_{A} \cap \mathbb{Z}\sp{n}$. 
\end{proposition}
\begin{proof}
The proof is given in Appendix \ref{SChilbertproof}.
\end{proof}

For two nonnegative integer vectors $x, y \in \ZZ_{+}^{n}$,
we write $y \preceq_{f} x$ if $y \leq x$ and $f(y) \leq f(x)$.
Note that $y \preceq_{f} x$ if and only if  
$(y,f(y)) \leq (x,f(x))$ in $\RR^{n} \times (\RR \cup \{ +\infty \})$.
We say that $x \in \ZZ_{+}^{n}$ is {\em $f$-minimal} if 
$x \in \dom f$ and
there exists no $y \in \ZZ_{+}^{n}$
such that $y \preceq_{f} x$ and $y \not = x$.  
That is%
\footnote{
$x$ is $f$-minimal if and only if \ 
$\argmin f_{[ \veczero,x ]} = \{ x \}$
for the function 
$ f_{[\veczero,x]}(y) = \left\{ \begin{array}{cl}
  f(y) & (y \in [\veczero,x]_{\ZZ}) , \\
  +\infty & (y \in \ZZ^{n} \setminus [\veczero,x]_{\ZZ}) .
  \end{array}\right.$
}, 
$x$ is $f$-minimal if and only if
it is the unique minimizer of the function $f$ 
restricted to the integer interval $[\bm{0},x]_{\ZZ}$.

The goal of this section is to establish the following
connection between $f$-minimality and the integer cone $C_{A}$
based at $\alpha \unitvec{A}$.

\begin{proposition}   \label{PRminimalnotincone}
Assume $\alpha$-local minimality {\rm (\ref{assumption1})}.
If $y \in \ZZ_{+}^{n}$ is $f$-minimal, then  
$y \not\in \alpha \unitvec{A} + C_{A}$
for any $A (\neq \emptyset) \subseteq N$. 
\finbox
\end{proposition}

Our proof of this proposition is based on several lemmas.

\begin{lemma} \label{LMtam1}
Assume $\alpha$-local minimality {\rm (\ref{assumption1})}.
For any $A\;(\neq \emptyset) \subseteq N$ and 
$\lambda \in \ZZ_{+}$
we have $(\alpha-1) \unitvec{A} \preceq_{f} (\alpha-1)\unitvec{A} +\lambda \unitvec{A}$.
\end{lemma}
\begin{proof}
First note that 
$(\alpha-1) \unitvec{A} \leq (\alpha-1)\unitvec{A} +\lambda \unitvec{A}$
for all $\lambda \in \ZZ_{+}$.
By integral convexity of $f$, 
$g(\lambda) = f(\lambda \unitvec{A})$
is a discrete convex function in $\lambda \in \ZZ_{+}$, and therefore,
\[
 g(\alpha-1) \leq \frac{\alpha-1}{\alpha} g(\alpha) +  \frac{1}{\alpha} g(0) .
\]
On the other hand,  $g(0) \leq g(\alpha)$
by the assumed $\alpha$-local minimality {\rm (\ref{assumption1})}.
Hence we have $g(\alpha-1) \leq  g(\alpha)$.
Since $g(0) < +\infty$, by discrete convexity of $g$, this implies
$g(\alpha-1) \leq g((\alpha-1)+\lambda)$ 
for all $\lambda \in \ZZ_{+}$, i.e.,
$f((\alpha-1) \unitvec{A}) \leq f((\alpha-1)\unitvec{A} +\lambda \unitvec{A})$ 
for all $\lambda \in \ZZ_{+}$.
\end{proof}

\begin{lemma}  \label{LMtam2}
Let $x \in \dom f$, $A\;(\neq \emptyset) \subseteq N$,
and assume $x \preceq_{f}  x + \unitvec{A}$.
Then for any $i \in N$, 
$\delta \in \{ +1, 0, -1 \}$, 
and $\lambda \in \ZZ_{+}$ we have
$x+\unitvec{A}+ \delta \unitvec{i}  \preceq_{f} 
(x + \unitvec{A} +  \delta \unitvec{i})  + \lambda \unitvec{A}$.
\end{lemma}
\begin{proof}
First note that 
$x+\unitvec{A}+ \delta \unitvec{i} 
\leq (x + \unitvec{A} +  \delta \unitvec{i}) + \lambda \unitvec{A}$.
We only need to show 
$f(x+\unitvec{A}+ \delta \unitvec{i}) 
\leq f((x + \unitvec{A} +  \delta \unitvec{i}) + \lambda \unitvec{A})$
when
$f((x + \unitvec{A} +  \delta \unitvec{i}) + \lambda \unitvec{A}) < +\infty$.
By integral convexity of $f$ we have
\begin{align*}
& \frac{1}{\lambda +1}
f((x + \unitvec{A} +  \delta \unitvec{i})  + \lambda \unitvec{A})
 +
 \frac{\lambda}{\lambda +1}  f(x) 
\\
& \geq
\tilde f(x + \unitvec{A} + \frac{\delta}{\lambda +1}  \unitvec{i})
\\ & 
= \frac{1}{\lambda +1}
f(x + \unitvec{A} +  \delta \unitvec{i})
 +
 \frac{\lambda}{\lambda +1}  f(x + \unitvec{A}) ,
\end{align*}
whereas
$f(x+\unitvec{A}) \geq f(x)$
by the assumption.
Hence
$f((x + \unitvec{A} +  \delta \unitvec{i}) + \lambda \unitvec{A})
\geq 
f(x+\unitvec{A}+ \delta \unitvec{i})$.
\end{proof}

\begin{lemma}  \label{LMdominatearound}
Let  $x \in \dom f$, $A\;(\neq \emptyset) \subseteq N$,
and assume $x \preceq_{f}  x + \unitvec{A}$.
For any $\lambda \in \ZZ_{+}$,
$\mu_{i}^{+},\mu_{i}^{-} \in \ZZ_{+} \;(i \in A)$,
and $\mu_{i}^{\circ} \in \ZZ_{+}\;(i \in N \setminus A)$,
the point

\begin{equation} \label{yx1A1iA}
 y = x + \unitvec{A}
  + \sum_{i \in A} \mu_{i}^{+}(\unitvec{A} + \unitvec{i}) 
  + \sum_{i \in A} \mu_{i}^{-} (\unitvec{A} - \unitvec{i})
  + \sum_{i \in N \setminus A} \mu_{i}^{\circ} (\unitvec{A}+ \unitvec{i})
  + \lambda \unitvec{A}
\end{equation}
is not $f$-minimal.
\end{lemma}
\begin{proof}
By the definition of an $f$-minimal point,
we assume $y \in \dom f$; since otherwise we are done.
Define
\begin{equation} \label{mu1A1i}
\mu = \sum_{i \in A} (\mu_{i}^{+}+\mu_{i}^{-}) 
      + \sum_{i \in N \setminus A} \mu_{i}^{\circ} ,
\end{equation}
which serves, in our proof, as an index to measure the distance between $x$ and $y$.
If $\mu \leq 1$,
then $y$ is not $f$-minimal
by Lemma~\ref{LMtam2}.
Suppose that $\mu \geq  2$.
In the following we construct a vector 
$x'$ such that
$x' \in \dom f$, $x' \preceq_{f}  x' + \unitvec{A}$,
$y$ is represented as (\ref{yx1A1iA}) with $x'$ 
in place of $x$,
and the index $\mu'$ for that representation
is strictly smaller than $\mu$.

Define $\beta = \mu + \lambda + 1$ and
\begin{align*}
A^{+} 
& = \{ i \in A \mid \mu_{i}^{+} \geq 1 \},
\\
A^{-} 
& = \{ i \in A \mid \mu_{i}^{-} \geq 1 \},
\\
A^{=}
& = \{ i \in A \mid \mu_{i}^{+} = \mu_{i}^{-} = 0 \},
\\
A^{\circ} 
& = \{ i \in N \setminus A \mid \mu_{i}^{\circ} \geq 1 \},
\end{align*}
where we may assume, without loss of generality, that
$A^{+} \cap A^{-} = \emptyset$. 
Then (\ref{yx1A1iA}) can be rewritten as
\[
y = x + \sum_{i \in A^{+}} \mu_{i}^{+}\unitvec{i} 
  - \sum_{i \in A^{-}} \mu_{i}^{-} \unitvec{i}
  + \sum_{i \in A^{\circ}} \mu_{i}^{\circ} \unitvec{i} + \beta \unitvec{A},
\]
which shows
\[
(y-x)_{i} =
\left\{
 \begin{array}{ll}
  \beta + \mu_{i}^{+} & (i \in A^{+}), \\
  \beta - \mu_{i}^{-} & (i \in A^{-}), \\
  \beta                   & (i \in A^{=}), \\
   \mu_{i}^{\circ} & (i \in A^{\circ}), \\
   0 & (\mbox{otherwise}) .
  \end{array} \right.
\]
Consider the point
\[
 z = \frac{\beta-1}{\beta} \  x + \frac{1}{\beta} \  y ,
\]
which is not contained in $\ZZ^{n}$ since
$ z = x + (y-x)/\beta$
and
$\displaystyle
1 \leq \max \big(
\max_{i \in A^{+}} \mu_{i}^{+} ,
\max_{i \in A^{-}} \mu_{i}^{-} ,
\max_{i \in A^{\circ}} \mu_{i}^{\circ} \big)
\leq \beta - 1$
with $A^{+} \cup A^{-} \cup A^{\circ} \not= \emptyset$.
Since $f$ is integrally convex and $x, y \in \dom f$,
we have
$\tilde{f}(z) \leq ((\beta-1)/\beta) f(x) + (1/\beta) f(y) < +\infty$.
On the other hand, since
\[
(z-x)_{i} =
\left\{
 \begin{array}{ll}
  1 + ({\mu_{i}^{+}}/{\beta}) & (i \in A^{+}), \\
  1 - ({\mu_{i}^{-}}/{\beta}) & (i \in A^{-}), \\
  1                   & (i \in A^{=}), \\
   {\mu_{i}^{\circ}}/{\beta} & (i \in A^{\circ}), \\
   0 & (\mbox{otherwise}) ,
  \end{array} \right.
\]
the integral neighborhood $N(z)$ of $z$ 
consists of all points
$x'$ that can be represented as 
\begin{equation} \label{xprimxADAD}
 x' = x + \unitvec{A} + \unitvec{A^{+} \cap D} - \unitvec{A^{-} \cap D}
+ \unitvec{A^{\circ} \cap D}
\end{equation}
for a subset $D$ of $A^{+} \cup A^{-} \cup A^{\circ}$.
Since $\tilde{f}(z) < +\infty$  and $z \not\in \ZZ\sp{n}$,
we must have $|N(z) \cap \dom f| \geq 2$,
which implies that 
there exists a nonempty $D$ for which $x' \in \dom f$.
Take such $D$ that is minimal with respect to set inclusion.

We claim that 
$x' \preceq_{f} x' + \unitvec{A}$.
Obviously we have $x' \leq x' + \unitvec{A}$.
To show
$f(x') \leq f(x' + \unitvec{A})$,
we may assume 
$x' + \unitvec{A} \in \dom f$.
Then we have the following chain of inequalities:
\begin{align*}
& f(x) + f(x' +\unitvec{A}) 
\\
& =
f(x) +  f(x + 2\unitvec{A} + \unitvec{A^{+} \cap D} - \unitvec{A^{-} \cap D}
+ \unitvec{A^{\circ} \cap D}) 
 \\ &\geq
 2 \tilde{f}(x + \unitvec{A} + \frac{1}{2}\unitvec{A^{+} \cap D} 
 - \frac{1}{2}\unitvec{A^{-} \cap D}
 + \frac{1}{2}\unitvec{A^{\circ} \cap D})
& [\mbox{by integral convexity of $f$}] 
\\ &=
f(x + \unitvec{A}) + f(x + \unitvec{A} + \unitvec{A^{+} \cap D} - \unitvec{A^{-} \cap D}
+ \unitvec{A^{\circ} \cap D})
& [\mbox{by minimality of $D$}] 
\\&= f(x + \unitvec{A}) + f(x')
&  [\mbox{by (\ref{xprimxADAD}) }] 
\\& \geq f(x) + f(x')
& [\mbox{by $x \preceq_{f} x + \unitvec{A}$}] 
\end{align*}
which shows $f(x' + \unitvec{A}) \geq f(x')$.
Therefore, $x' \preceq_{f} x' + \unitvec{A}$ is true.

We finally consider the index (\ref{mu1A1i}) associated with $x'$,
which we denote by $\mu'$.
The substitution of (\ref{xprimxADAD}) into (\ref{yx1A1iA}) yields
\begin{align}
 y &= x + \unitvec{A}
 + \sum_{i \in A^{+}} \mu_{i}^{+}(\unitvec{A} + \unitvec{i}) 
 + \sum_{i \in A^{-}} \mu_{i}^{-} (\unitvec{A} - \unitvec{i})
   + \sum_{i \in A^{\circ}} \mu_{i}^{\circ} (\unitvec{A}+ \unitvec{i})
   + \lambda \unitvec{A} \nonumber
\\ & = x' + \unitvec{A} + \left(
  \sum_{i \in A^{+} \setminus D} \mu_{i}^{+}(\unitvec{A} + \unitvec{i}) +
  \sum_{i \in A^{+} \cap D} (\mu_{i}^{+}-1)(\unitvec{A} + \unitvec{i}) \right) \nonumber
\\ &\quad + \left(
  \sum_{i \in A^{-} \setminus D} \mu_{i}^{-}(\unitvec{A} - \unitvec{i}) +
  \sum_{i \in A^{-} \cap D} (\mu_{i}^{-}-1)(\unitvec{A} - \unitvec{i}) \right) \nonumber
\\ &\quad + \left(
  \sum_{i \in A^{\circ} \setminus D} \mu_{i}^{\circ}(\unitvec{A} + \unitvec{i}) +
  \sum_{i \in A^{\circ} \cap D} (\mu_{i}^{\circ}-1)(\unitvec{A} + \unitvec{i}) \right)
  + (\lambda + |D| -1) \unitvec{A}. 
\label{LMdominatearound-eq1}
\end{align}
This shows $\mu' = \mu - |D| \leq \mu -1$.

The above procedure finds  
$x' \in \dom f$ such that $x' \preceq_{f}  x' + \unitvec{A}$
and $\mu'\leq \mu -1$,
when given
$x \in \dom f$ such that $x \preceq_{f}  x + \unitvec{A}$
and $\mu \geq 2$.
By repeated application of this procedure
we can eventually arrive at 
$x'' \in \dom f$ such that $x'' \preceq_{f}  x'' + \unitvec{A}$
and $\mu''\leq 1$.
Then $y$ is not $f$-minimal by Lemma~\ref{LMtam2} for $x''$.
\end{proof}

\begin{lemma} \label{LMtam5}
If $(\alpha \unitvec{A} + C_A) \cap (\dom f) \neq \emptyset$,
then $\alpha \unitvec{A} \in \dom f$.
\end{lemma}
\begin{proof}
To prove by contradiction,  take 
$y \in (\alpha \unitvec{A} + C_A) \cap (\dom f)$
that is  minimal with respect to the vector ordering (componentwise ordering)
and assume that
$y \neq \alpha \unitvec{A}$.
The vector $y$ can be represented as 
\[
 y = \alpha \unitvec{A}+
\sum_{i \in A} \mu_{i}^{+}(\unitvec{A} + \unitvec{i}) + \sum_{i \in A} \mu_{i}^{-} (\unitvec{A} - \unitvec{i})
   + \sum_{i \in N \setminus A} \mu_{i}^{\circ} (\unitvec{A}+ \unitvec{i})+ \lambda \unitvec{A} 
\]
with some
$\mu_{i}^{+},\mu_{i}^{-} \in \ZZ_{+} \;(i \in A)$,
$\mu_{i}^{\circ} \in \ZZ_{+}\;(i \in N \setminus A)$,
and $\lambda \in \ZZ_{+}$, where
\[
 \beta = \alpha + \sum_{i \in A^{+}}\mu_{i}^{+}+
  \sum_{i \in A^{-}} \mu_{i}^{-} + \sum_{i \in A^{\circ}} \mu_{i}^{\circ} + \lambda 
\]
is strictly larger than $\alpha$
since $y \neq \alpha \unitvec{A}$.
Define
\begin{align*}
A^{+} 
& = \{ i \in A \mid \mu_{i}^{+} \geq 1 \},
\\
A^{-} 
& = \{ i \in A \mid \mu_{i}^{-} \geq 1 \},
\\
A^{=}
& = \{ i \in A \mid \mu_{i}^{+} = \mu_{i}^{-} = 0 \},
\\
A^{\circ} 
& = \{ i \in N \setminus A \mid \mu_{i}^{\circ} \geq 1 \},
\end{align*}
where we may assume, without loss of generality, that
$A^{+} \cap A^{-} = \emptyset$ and $A^{-} \not= A$.
We have
\[
y_{i} =
\left\{
 \begin{array}{ll}
  \beta + \mu_{i}^{+} & (i \in A^{+}), \\
  \beta - \mu_{i}^{-} & (i \in A^{-}), \\
  \beta                   & (i \in A^{=}), \\
   \mu_{i}^{\circ} & (i \in A^{\circ}), \\
   0 & (\mbox{otherwise}) .
  \end{array} \right.
\]
Consider the point
\[
 z = \frac{\beta-1}{\beta} \ y + \frac{1}{\beta} \  \veczero .
\]
Since $f$ is integrally convex and $y, \veczero \in \dom f$,
we have
$\tilde{f}(z) \leq ((\beta-1)/\beta) f(y) + (1/\beta) f(\veczero ) 
< +\infty$.
Note that
\[
z_{i} =
\left\{
 \begin{array}{ll}
  (\beta - 1) + \mu_{i}^{+} - \mu_{i}^{+}/\beta & (i \in A^{+}), \\
  (\beta - 1) - \mu_{i}^{-} + \mu_{i}^{-}/\beta & (i \in A^{-}), \\
  (\beta - 1)                   & (i \in A^{=}), \\
   \mu_{i}^{\circ} - \mu_{i}^{\circ}/\beta & (i \in A^{\circ}), \\
   0 & (\mbox{otherwise}).
  \end{array} \right.
\]
If $A^{+} \cup A^{-} \cup A^{\circ} = \emptyset$,
we are done with a contradiction. Indeed, we then have
$z = (\alpha  + \lambda -1 ) \unitvec{A}$
and
$y = (\alpha  + \lambda ) \unitvec{A}$,
and hence 
$z \leq y$, $z \not= y$,
and $z \in \dom f$ by $f(z)=\tilde{f}(z) < +\infty$.
In the following we assume 
$A^{+} \cup A^{-} \cup A^{\circ} \not= \emptyset$,
which implies $z \not\in \ZZ\sp{n}$.

The integral neighborhood $N(z)$ of $z$ 
consists of all points
$y'$ that can be represented as 
\begin{align*}
 y' &=  (\beta -1)\unitvec{A} + \left(
  \sum_{i \in A^{+} \setminus D} \mu_{i}^{+} \unitvec{i} +
  \sum_{i \in A^{+} \cap D} (\mu_{i}^{+}-1) \unitvec{i} \right) 
\nonumber \\ 
&\quad - \left(
  \sum_{i \in A^{-} \setminus D} \mu_{i}^{-} \unitvec{i} +
  \sum_{i \in A^{-} \cap D} (\mu_{i}^{-}-1) \unitvec{i} \right) 
+ \left(
  \sum_{i \in A^{\circ} \setminus D} \mu_{i}^{\circ} \unitvec{i} +
  \sum_{i \in A^{\circ} \cap D} (\mu_{i}^{\circ}-1) \unitvec{i} \right)
\end{align*}
for a subset $D$ of $A^{+} \cup A^{-} \cup A^{\circ}$.
Since $\tilde{f}(z) < +\infty$  and $z \not\in \ZZ\sp{n}$,
we must have $|N(z) \cap \dom f| \geq 2$,
which implies that 
there exists a nonempty $D$ for which $y' \in \dom f$.
Take any $y' \in N(z) \cap \dom f$ with $D \not= \emptyset$.
Then  
$y' \leq y$ and $y' \neq y$, since $A^{-} \not= A$ and
\begin{align*}
 y' - y = - \unitvec{A} 
 - \sum_{i \in A^{+} \cap D} \unitvec{i} 
 +  \sum_{i \in A^{-} \cap D} \unitvec{i} 
 -  \sum_{i \in A^{\circ} \cap D} \unitvec{i} \leq \veczero .
\end{align*}
We also have $y' \in (\alpha \unitvec{A} + C_{A})$
by an alternative expression of $y'$:
\begin{align*}
 y' 
&=  \alpha \unitvec{A} 
  + \left(
  \sum_{i \in A^{+} \setminus D} \mu_{i}^{+}(\unitvec{A} + \unitvec{i}) +
  \sum_{i \in A^{+} \cap D} (\mu_{i}^{+}-1)(\unitvec{A} + \unitvec{i}) 
    \right) 
\nonumber
\\ &\quad + \left(
  \sum_{i \in A^{-} \setminus D} \mu_{i}^{-}(\unitvec{A} - \unitvec{i}) +
  \sum_{i \in A^{-} \cap D} (\mu_{i}^{-}-1)(\unitvec{A} - \unitvec{i}) \right) 
\nonumber
\\ &\quad 
 + \left(
  \sum_{i \in A^{\circ} \setminus D} \mu_{i}^{\circ}(\unitvec{A} + \unitvec{i}) +
  \sum_{i \in A^{\circ} \cap D} (\mu_{i}^{\circ}-1)(\unitvec{A} + \unitvec{i})
 \right)
+  (\lambda +|D| -1)\unitvec{A} .
\end{align*}
Hence $y' \in (\alpha \unitvec{A} + C_{A}) \cap (\dom f)$,
a contradiction to the minimality of $y$.
\end{proof}

We are now in the position to prove Proposition \ref{PRminimalnotincone}. 
To prove the contrapositive of the claim, 
suppose that  $y \in \alpha \unitvec{A} + C_{A}$ for some $A$.
Then $y$ can be expressed as
\[
y = \alpha \unitvec{A} 
 + \sum_{i \in A} \mu_{i}^{+}(\unitvec{A} + \unitvec{i})
    + \sum_{i \in A} \mu_{i}^{-} (\unitvec{A} - \unitvec{i})
   + \sum_{i \in N \setminus A} \mu_{i}^{\circ} (\unitvec{A}+ \unitvec{i})
   +  \lambda \unitvec{A} 
\]
for some $\mu_{i}^{+},\mu_{i}^{-}, \mu_{i}^{\circ}, \lambda \in \ZZ_{+}$.
Equivalently,
\[
y = ((\alpha-1) \unitvec{A} + \unitvec{A}) + 
    \sum_{i \in A} \mu_{i}^{+}(\unitvec{A} + \unitvec{i})
    + \sum_{i \in A} \mu_{i}^{-} (\unitvec{A} - \unitvec{i})
   + \sum_{i \in N \setminus A} \mu_{i}^{\circ} (\unitvec{A}+ \unitvec{i})
   + \lambda \unitvec{A} ,
\]
which corresponds to the right-hand side of (\ref{yx1A1iA})
with $x = (\alpha-1) \unitvec{A}$. 
By Lemma~\ref{LMtam5}, we have $\alpha \unitvec{A} \in \dom f$.
Since 
$x = (\alpha-1) \unitvec{A} \preceq_{f} \alpha \unitvec{A} =  x + \unitvec{A}$
 by Lemma~\ref{LMtam1},
Lemma~\ref{LMdominatearound} shows that $y$ is not $f$-minimal.
This completes the proof of Proposition \ref{PRminimalnotincone}.

\subsection{Proof of Proposition \ref{PRproxintcnvnonneg} for $n=2$}
\label{SCproxproofn2}

In this section we prove Proposition~\ref{PRproxintcnvnonneg} for $n=2$
as an illustration of the proof method using the tools
introduced in Section~\ref{SCprooftools}.
This also gives an alternative proof of Theorem~\ref{THproxintcnvdim2alpha}.

Recall that 
$\dom f$ is assumed to be a bounded subset of $\ZZ_{+}^{2}$,
which implies, in particular, that $\argmin f \not= \emptyset$.
Take
$x^{*}=(x_{1}^{*},x_{2}^{*})  \in \argmin f$
that is $f$-minimal.
We may assume
$x_{1}^{*} \geq x_{2}^{*}$ 
by Proposition~\ref{PRintcnvinvar} (2).
Since $x^{*}$ is $f$-minimal, 
Proposition~\ref{PRminimalnotincone}
shows that $x^{*}$ belongs to
\[
 X^{*} = 
  \{  (x_{1},x_{2}) \in \ZZ_{+}^{2} \mid  x_{1} \geq x_{2} \}
 \setminus 
  \left(  (\alpha \unitvec{A} + C_{A}) \cup (\alpha \unitvec{N} + C_{N})  \right),
\]
where $A = \{ 1 \}$ and $N = \{ 1, 2 \}$. 
On noting 
\begin{align*}
C_{A} &= \{ \mu_{1} (1,0) + \mu_{12} (1,1) \mid \mu_{1},\mu_{12} \in \ZZ_{+} \},
\\
C_{N} &= \{ \mu_{1} (1,0) + \mu_{2} (0,1) \mid \mu_{1},\mu_{2} \in \ZZ_{+} \},
\end{align*}
we see that 
$X^{*}$ consists of all integer points contained in the parallelogram
with vertices $(0,0)$, 
$(\alpha -1, 0)$,
$(2\alpha -2, \alpha -1)$,
$(\alpha -1, \alpha -1)$.
Therefore,
$ \| x^{*}\|_{\infty} \leq 2(\alpha -1)$.
Thus Proposition \ref{PRproxintcnvnonneg} for $n=2$ is proved.

\subsection{Proof of Proposition \ref{PRproxintcnvnonneg} for $n \geq 3$}
\label{SCproxproofn3}

In this section we prove Proposition~\ref{PRproxintcnvnonneg} for $n \geq 3$
by induction on $n$.
Accordingly we assume that 
Proposition~\ref{PRproxintcnvnonneg} is true for every integrally convex function
in $n-1$ variables.

Let $f: \mathbb{Z}^{n} \to \mathbb{R} \cup \{ +\infty  \}$
be an integrally convex function such that
$\dom f$ is a bounded subset of\, $\ZZ_{+}^{n}$ containing the origin $\bm{0}$.
Note that $\argmin f \not= \emptyset$ and take
$x^{*}=(x_{1}^{*},x_{2}^{*},\ldots,x_{n}^{*})  \in \argmin f$
that is $f$-minimal.
Then 
\begin{equation}\label{assumption3}
  [\veczero,x^{*}]_{\ZZ} \cap \argmin f = \{x^{*}\} .
\end{equation}
We may assume
\begin{equation}\label{assumption2}
 x_{1}^{*} \geq x_{2}^{*} \geq \cdots \geq x_{n}^{*} 
\end{equation}
by Proposition~\ref{PRintcnvinvar} (2).

The following lemma reveals a significant property
of integrally convex functions that will be used here for induction on $n$.
Note that, by (\ref{assumption3}), 
$x^{*}$ satisfies the condition imposed on $x^{\bullet}$.   

\begin{lemma}  \label{LMtam4}
Let $x^{\bullet} \in \dom f$ be an $f$-minimal point.
Then for any $i \in N$ with $x^{\bullet}_i \geq 1$ 
there exists an $f$-minimal point $x^{\circ} \in \dom f$ such that
\[
 \veczero \leq x^{\circ} \leq x^{\bullet},
\quad 
 \| x^{\circ}-x^{\bullet} \|_\infty = x^{\bullet}_{i} -  x^{\circ}_{i}  = 1.
 \]
\end{lemma}
\begin{proof}
Let $x^{\circ}$ be
a minimizer of $f(x)$ among those $x$ which belong to
$X = \{ x \in \ZZ^{n} \mid \veczero \leq x \leq x^{\bullet}, \ 
\| x-x^{\bullet} \|_\infty = 1, \  x_{i} = x^{\bullet}_{i} -1 \}$;
in case of multiple minimizers, 
we choose a minimal minimizer 
with respect to the vector ordering 
(componentwise ordering).
To prove  $f$-minimality of $x^{\circ}$,
suppose, to the contrary, that 
there exists
$z \in [\veczero,x^{\circ}]_{\ZZ} \setminus \{x^{\circ}\}$
with $f(z) \leq f(x^{\circ})$.
We have
$\ell = \| z-x^{\bullet} \|_\infty \geq 2$,
since otherwise $z \in X$ and this contradicts the minimality of $x^{\circ}$.

Consider $y \in \RR_{+}^{n}$ defined by 
\begin{equation}\label{prop4eq1}
  y = \frac{\ell-1}{\ell} x^{\bullet} + \frac{1}{\ell}z .
\end{equation}
The value of the local convex extension  $\tilde{f}$ of $f$ at $y$ can be represented as
\[
  \tilde{f}(y) = \sum_{y^{j} \in Y}\lambda_{j} f(y^{j})
\]
with some set $Y \subseteq N(y) \cap \dom f$ and positive
coefficients $\lambda_{j}$ such that
\begin{equation}\label{prop4eq25}
y = \sum_{y^{j} \in Y}\lambda_{j} y^{j},
\qquad 
\sum_{y^{j} \in Y} \lambda_{j} = 1.
\end{equation}
Since $\| y-x^{\bullet} \|_\infty = 1$ and $y \leq x^{\bullet}$
by (\ref{prop4eq1}),
either 
$y^{j}_{i} = x^{\bullet}_{i}-1$
or $y^{j}_{i} = x^{\bullet}_{i} $
holds for each  $y^{j} \in Y$.
Define
\[
Y^{<} = \{ y^{j} \in Y \mid y^{j}_{i} = x^{\bullet}_{i}-1 \},
\qquad
Y^{=} = \{ y^{j} \in Y \mid y^{j}_{i} = x^{\bullet}_{i} \}.
\]
Then we see
\begin{equation}\label{prop4eq3}
 \sum_{y^{j} \in Y^{<}} \lambda_{j} = \frac{x^{\bullet}_{i}-z_{i}}{\ell},
\qquad
 \sum_{y^{j} \in Y^{=}} \lambda_{j} = 1- \frac{x^{\bullet}_{i}-z_{i}}{\ell}
\end{equation}
from
\[
 y_{i} =  \frac{\ell-1}{\ell} x^{\bullet}_{i} + \frac{1}{\ell}z_{i} = 
  x^{\bullet}_{i} - \frac{x^{\bullet}_{i}-z_{i}}{\ell} .
\]

On the other hand, we have 
\begin{equation}\label{prop4intcnv}
  \tilde{f}(y) \leq \frac{\ell-1}{\ell} f(x^{\bullet}) + \frac{1}{\ell}f(z) 
\end{equation}
by integral convexity of $f$.
We divide into cases to derive a contradiction to this inequality.

Case 1 ($x^{\bullet}_{i}-z_{i} = \ell$):
We have $Y = Y^{<}$ 
by (\ref{prop4eq25}) and (\ref{prop4eq3}) and then 
$f(x^{\circ}) \leq f(y^{j})$ for all $y^{j} \in Y$
by the definition of $x^{\circ}$.
Hence 
\begin{equation}\label{prop4case1a}
   f(x^{\circ}) \leq  \sum_{y^{j} \in Y}\lambda_{j} f(y^{j}) = \tilde{f}(y) .
\end{equation}
For the right-hand side of (\ref{prop4intcnv}), 
note first that
 the $f$-minimality of $x^{\bullet}$
and $x^{\circ} \in [\veczero,x^{\bullet}]_{\ZZ} \setminus \{x^{\bullet}\}$
imply $f(x^{\bullet}) < f(x^{\circ})$.
Then it follows from 
$f(x^{\bullet}) < f(x^{\circ})$ and $f(z) \leq f(x^{\circ})$
that 
\begin{equation}\label{prop4case1b}
\frac{\ell-1}{\ell} f(x^{\bullet}) + \frac{1}{\ell}f(z) <   f(x^{\circ}) .
\end{equation}
But (\ref{prop4case1a}) and (\ref{prop4case1b}) together
contradict (\ref{prop4intcnv}).

Case 2 ($x^{\bullet}_{i}-z_{i} < \ell$):
In this case $Y^{=}$ is nonempty.
Since $x^{\bullet} \not\in Y$ by $\| x^{\bullet} - y \|_\infty = 1$,
every $y^{j} \in Y$ is distinct from $x^{\bullet}$,
whereas 
$y^{j} \in [\veczero,x^{\bullet}]_{\ZZ}$.
Then the assumed $f$-minimality of $x^{\bullet}$ implies
\begin{equation} \label{prm2}
f(y^{j}) > f(x^{\bullet}) 
\qquad (\forall \, y^{j} \in Y=Y^{=} \cup Y^{<}) .
\end{equation}
We also have
\begin{equation} \label{prm1}
f(y^{j}) \geq f(x^{\circ}) \geq f(z) 
\qquad (\forall \, y^{j} \in Y^{<}) ,
\end{equation}
which is obvious from the definitions of $x^{\circ}$ and $z$.
Then we have
\begin{align*}
 \tilde{f}(y) &=
 \sum_{y^{j} \in Y^{=}} \lambda_{j} f(y^{j})
+ \sum_{y^{j} \in Y^{<}} \lambda_{j} f(y^{j}) 
\\ & & [\mbox{by (\ref{prm2}), (\ref{prm1}), $Y^{=} \not= \emptyset$}] 
\\ &> 
\sum_{y^{j} \in Y^{=}} \lambda_{j}  f(x^{\bullet}) + 
 \sum_{y^{j} \in Y^{<}} \lambda_{j} f(z) 
\\ & & [\mbox{by (\ref{prop4eq3})}]
\\ &=
(1-\frac{x^{\bullet}_{i}-z_{i}}{\ell}) f(x^{\bullet})
 + \frac{x^{\bullet}_{i}-z_{i}}{\ell}f(z) 
\\ & & [\mbox{by \ } 
 \frac{x^{\bullet}_{i}-z_{i}}{\ell} \geq \frac{1}{\ell}, \   f(x^{\bullet}) \leq f(z)]
\\ & \geq
\frac{\ell-1}{\ell} f(x^{\bullet}) + \frac{1}{\ell}f(z) .
\end{align*}
This is a contradiction to (\ref{prop4intcnv}).
\end{proof}

Lemma~\ref{LMtam4} can be applied repeatedly, since the resulting point
$x^{\circ}$ satisfies the condition imposed on the initial point $x^{\bullet}$.   
Starting with
$x^{\bullet} = x^{*}$
we apply Lemma~\ref{LMtam4} repeatedly with $i = n$.
After $x_{n}^{*}$ applications, we arrive at a point
$\hat{x} = (\hat{x}_{1},\hat{x}_{2},\ldots,\hat{x}_{n-1},0)$.
This point $\hat{x}$ is $f$-minimal and
\begin{equation}\label{txjxjxn}
x_{j}^{*} - x_{n}^{*} \leq  \hat{x}_{j} 
\qquad (j=1,2,\ldots,n-1).
\end{equation}

We now consider a function $\hat{f}: \ZZ^{n-1} \to \RR \cup \{+\infty\}$
defined by 
\[
\hat{f}(x_{1},x_{2},\ldots,x_{n-1}) = 
\left\{ 
\begin{array}{ll}
f(x_{1},x_{2},\ldots,x_{n-1},0)
& (0 \leq x_{j} \leq \hat{x}_{j} \ (j=1,2,\ldots,n-1) ), 
\\
+\infty
& \mbox{(otherwise)}.
\end{array} \right.
\]
This function $\hat{f}$ is an integrally convex function
in $n-1$ variables,
and the origin $\veczero$ is $\alpha$-local minimal for $\hat{f}$.
By the induction hypothesis, we can apply Proposition~\ref{PRproxintcnvnonneg} to $\hat{f}$
to obtain
\begin{equation}\label{proxbdn1}
 \| \hat{x} \|_{\infty} \leq \beta_{n-1} (\alpha - 1).
\end{equation}
Note that $\hat{x}$ is the unique minimizer of $\hat{f}$.

Combining (\ref{txjxjxn})  and (\ref{proxbdn1}) we obtain
\begin{equation}\label{eq1}
 x_{1}^{*} - x_{n}^{*} \leq \beta_{n-1} (\alpha - 1) .
\end{equation}
We also have
\begin{equation}\label{eqM2a}
  x_{n}^{*} \leq \frac{n-1}{n+1} x_{1}^{*}  + \frac{2(\alpha - 1)}{n+1} 
\end{equation}
as a consequence of $f$-minimality of $x^{*}$; see Lemma~\ref{LMdifsum} below.
It follows from  (\ref{eq1}) and (\ref{eqM2a}) that
\[
 x_{1}^{*}  
\ \leq \ 
  x_{n}^{*} + \beta_{n-1} (\alpha - 1) 
\ \leq \ 
 \frac{n-1}{n+1} x_{1}^{*}  + \frac{2(\alpha - 1)}{n+1} + \beta_{n-1} (\alpha - 1) .
\]
This implies
\[ 
 x_{1}^{*}   \leq  
\left(\frac{n+1}{2} \beta_{n-1} + 1 \right) (\alpha - 1)
= \beta_{n} (\alpha - 1),
\] 
where the recurrence relation
\[ 
 \beta_{n} =  \frac{n+1}{2} \beta_{n-1}  + 1
\] 
is used.

It remains to derive inequality (\ref{eqM2a}) from $f$-minimality of $x^{*}$.

\begin{lemma}  \label{LMdifsum}
The following inequalities hold for $x^{*}$ and $\alpha$.
\\
{\rm (1)} 
\vspace{-0.5\baselineskip}
\begin{equation}\label{coneineqn}
 \sum_{i=1}^{n-1}(x_{i}^{*} - x_{n}^{*}) \geq x_{n}^{*} - \alpha + 1.
\end{equation}
{\rm (2)} 
\vspace{-0.5\baselineskip}
\begin{equation}\label{coneineq1}
 \sum_{i=2}^{n}(x_{1}^{*} - x_{i}^{*}) \geq x_{1}^{*} - \alpha + 1.
\end{equation}
{\rm (3)} 
\vspace{-0.5\baselineskip}
\begin{equation}\label{eqM2}
  x_{n}^{*} \leq \frac{n-1}{n+1} x_{1}^{*}  + \frac{2(\alpha - 1)}{n+1} .
\end{equation}
\end{lemma}
\begin{proof}
(1)
To prove by contradiction, suppose that
 $\displaystyle \sum_{i=1}^{n-1}(x_{i}^{*} - x_{n}^{*}) \leq x_{n}^{*} - \alpha$.
Then the expression
\[
 x^{*} = \alpha \unitvec{N} + \sum_{i=1}^{n-1}(x_{i}^{*} - x_{n}^{*})(\unitvec{N}+\unitvec{i}) + 
 \left( x_{n}^{*}-\alpha-\sum_{i=1}^{n-1}(x_{i}^{*} - x_{n}^{*})\right) \unitvec{N}
\]
shows $x^{*} \in \alpha \unitvec{N} + C_{N}$.
By Proposition~\ref{PRminimalnotincone},
 this contradicts the fact that $x^{*}$ is $f$-minimal.

(2) 
To prove by contradiction, suppose that
$\displaystyle \sum_{i=2}^{n}(x_{1}^{*} - x_{i}^{*}) \leq x_{1}^{*} - \alpha$.
Then the expression
\[
 x^{*} = \alpha \unitvec{N} + \sum_{i=2}^{n}(x_{1}^{*} - x_{i}^{*})(\unitvec{N}-\unitvec{i}) + 
 \left( x_{1}^{*}-\alpha-\sum_{i=2}^{n}(x_{1}^{*} - x_{i}^{*})\right) \unitvec{N}
\]
shows $x^{*} \in \alpha \unitvec{N} + C_{N}$.
By Proposition~\ref{PRminimalnotincone},
 this contradicts the fact that $x^{*}$ is $f$-minimal.

(3)
Since
\[
 \sum_{i=1}^{n-1}(x_{i}^{*} - x_{n}^{*}) + \sum_{i=2}^{n}(x_{1}^{*} - x_{i}^{*}) = n (x_{1}^{*} - x_{n}^{*}) ,
\]
the addition of (\ref{coneineqn}) and (\ref{coneineq1}) yields
\[
 n (x_{1}^{*} - x_{n}^{*}) \geq x_{1}^{*} +x_{n}^{*} - 2(\alpha - 1) ,
\]
which is equivalent to (\ref{eqM2}).
\end{proof}

This completes the proof of
Proposition~\ref{PRproxintcnvnonneg}, and hence that of Theorem~\ref{THproxintcnv} (1).

\subsection{Estimation of $\beta_{n}$}
\label{SCproxproofbeta}

The estimate of $\beta_{n}$ 
given in Theorem~\ref{THproxintcnv} (2) is derived in this section.

The recurrence relation (\ref{eqMrec1}) can be rewritten as
\[
 \frac{2^{n}}{(n+1)!} \beta_{n} =
 \frac{2^{n-1}}{n!} \beta_{n-1} +  \frac{2^{n}}{(n+1)!} ,
\]
from which follows
\begin{align}
 \frac{2^{n}}{(n+1)!} \beta_{n} 
& =
 \frac{2^{2}}{3!} \beta_{2}
  +  \sum_{k=3}^{n}\frac{2^{k}}{(k+1)!} 
=   \frac{4}{3} +  \sum_{k=3}^{n}\frac{2^{k}}{(k+1)!} . 
\label{prfbbd1}
\end{align}
For the last term we have

\begin{align}
  \sum_{k=3}^{n}\frac{2^{k}}{(k+1)!} 
& \leq
  \sum_{k=3}^{7}\frac{2^{k}}{(k+1)!} 
 +  \sum_{k=8}^{\infty}\frac{2^{k}}{(k+1)!} 
 \leq
\frac{167}{315} 
\label{prfbbd2}
\end{align}
since 
\begin{align*}
  \sum_{k=3}^{7}\frac{2^{k}}{(k+1)!} 
& =
\frac{2^{3}}{4!} 
+ \frac{2^{4}}{5!} 
+ \frac{2^{5}}{6!} 
+ \frac{2^{6}}{7!} 
+ \frac{2^{7}}{8!}
 = \frac{166}{315}  
\ ( \approx 0.53),
\\
  \sum_{k=8}^{\infty}\frac{2^{k}}{(k+1)!} 
& \leq 
  \frac{2^{8}}{9!}  \sum_{k=1}^{\infty}\left( \frac{2}{10} \right)^{k-1} 
 = \frac{320}{9!} 
\ (\approx 8.8 \times 10^{-4} ) 
 < \frac{1}{315} . 
\end{align*}
Substitution of (\ref{prfbbd2}) into (\ref{prfbbd1}) yields
\begin{align*}
\beta_{n} 
& \leq 
 \left(   \frac{4}{3} + \frac{167}{315}     \right)
 \frac{(n+1)!}{2^{n}}  
= \frac{587}{315} \times
 \frac{(n+1)!}{2^{n}} 
\leq 
 \frac{(n+1)!}{2^{n-1}} .
\end{align*}
Thus the upper bound (\ref{eqMrec2est}) is proved.

\section{Optimization of Integrally Convex Functions}
\label{SCalgimpl}

In spite of the facts that the factor
$\beta_{n}$ of the proximity bound is superexponential in $n$ 
and that integral convexity is not stable under scaling,
we can design a proximity-scaling type algorithm
for minimizing integrally convex functions
with bounded effective domains.
The algorithm runs in 
$C(n) \log_{2} K_{\infty}$ time
for some constant $C(n)$ depending only on $n$,
where 
$K_{\infty}$ $(>0)$ 
denotes the $\ell_{\infty}$-size of the effective domain. 
This means that, if the dimension $n$ is fixed
and treated as a constant,
the algorithm is polynomial in the problem size.
Note that no algorithm 
for integrally convex function minimization
can be polynomial in $n$,
since any function  
on the unit cube $\{ 0,1 \}\sp{n}$ is integrally convex.

The proposed algorithm is a modification of the generic
proximity-scaling algorithm given in the Introduction.
In Step~S1, we replace  
the function $\tilde f(y) = f(x + \alpha y)$
with its restriction to the 
discrete rectangle  
$\{ y \in \mathbb{Z}\sp{n} \mid 
\| \alpha y \|_{\infty} \leq \beta_{n} (2\alpha - 1) \}$,
which is denoted by $\hat f(y)$.
Then a local minimizer of $\hat f(y)$ is found to update $x$ to $x+ \alpha y$. 
Note that a local minimizer of $\hat f(y)$ 
can be found, 
e.g., by any descent method (the steepest descent method, in particular).

\begin{tabbing}     
\= {\bf Proximity-scaling algorithm for integrally convex functions}%
\\
\> \quad  S0: 
   \= Find an initial vector $x$ with $f(x) < +\infty$, and set 
   $\alpha := 2\sp{\lceil \log_{2} K_{\infty} \rceil}$.
\\
\> \quad  S1:
   \>  Find an integer vector $y$  that locally minimizes  
\\ \> \> \quad 
      $\hat f(y) = 
   \left\{ \begin{array}{ll}  
  f(x + \alpha y) 
    & (\| \alpha y \|_{\infty} \leq \beta_{n} (2\alpha - 1) ) , \\
   +\infty & (\mbox{otherwise}) , \\
 \end{array}\right.
$
\\ \> \>
    in the sense of   
$\hat f(y) \leq \hat f(y +  d)$ ($\forall d \in  \{ -1, 0, +1 \}^{n}$)
\\ \> \>
   (e.g., by the steepest descent method),  
   and set $x:= x+ \alpha y$.  \\
\> \quad  S2: 
    \> If $\alpha = 1$, then stop \
       ($x$ is a minimizer of $f$).             
\\
\> \quad  S3: 
  \> Set  $\alpha:=\alpha/2$, and go to S1.  
\end{tabbing}

\begin{tabbing}     
\= {\bf The steepest descent method to locally minimize $\hat f(y)$}%
\\
\> \quad  D0: 
   \= Set $y := \bm{0}$.
\\
\> \quad  D1:
   \>  Find $d \in  \{ -1, 0, +1 \}^{n}$ that minimizes $\hat f(y +  d)$.
\\
\> \quad  D2: 
    \> If $\hat f(y) \leq \hat f(y +  d)$, then stop \
       ($y$ is a local minimizer of $\hat f$).             
\\
\> \quad  D3: 
  \> Set  $y:= y+d$, and go to D1.  
\end{tabbing}

The correctness of the algorithm can be shown as follows.
We first assume that $f$ has a unique (global) minimizer $x\sp{*}$.
Let $x\sp{2\alpha}$ denote 
the vector $x$ at the beginning of Step~S1, and define
\begin{align*}
 f\sp{(\alpha)}(x) &= 
   \left\{ \begin{array}{ll}  
  f(x)  & (\| x - x\sp{2\alpha} \|_{\infty} \leq \beta_{n} (2\alpha - 1) ) , \\
   +\infty & (\mbox{otherwise}) , \\
 \end{array}\right.
\\
\hat f\sp{(\alpha)}(y) &= 
   \left\{ \begin{array}{ll}  
  f(x\sp{2\alpha} + \alpha y) 
           & (\| \alpha y \|_{\infty} \leq \beta_{n} (2\alpha - 1) ) , \\
   +\infty & (\mbox{otherwise}) . \\
 \end{array}\right.
\end{align*}
Note that $f\sp{(\alpha)}$ is integrally convex, whereas
$\hat f\sp{(\alpha)}$ is not necessarily so. 
Let $\hat y\sp{\alpha}$ be the output of Step~S1
and 
$x\sp{\alpha} = x\sp{2\alpha} + \alpha \hat y\sp{\alpha}$.
Then
$\hat y\sp{\alpha}$ is a local minimizer of $\hat f\sp{(\alpha)}$
and
$x\sp{\alpha} - x\sp{2\alpha} 
= \alpha \hat y\sp{\alpha} \in (\alpha \mathbb{Z})\sp{n}$.

\begin{lemma}  \label{LMoptindomalpha}
$x\sp{*} \in \dom f\sp{(\alpha)}$
for all $\alpha$.
\end{lemma}
\begin{proof}
This is obviously true in the initial phase with 
$\alpha = 2\sp{\lceil \log_{2} K_{\infty} \rceil}$.
To prove 
$x\sp{*} \in \dom f\sp{(\alpha)}$
by induction on descending $\alpha$,
we show that
$x\sp{*} \in \dom f\sp{(\alpha)}$
implies 
$x\sp{*} \in \dom f\sp{(\alpha/2)}$.
Since 
$x\sp{*} \in \dom f\sp{(\alpha)}$
and 
$x\sp{*} \in \argmin f$, we have  
$x\sp{*} \in \argmin f\sp{(\alpha)}$.
On the other hand, $x\sp{\alpha}$ is an $\alpha$-local minimizer of $f\sp{(\alpha)}$, since 
$\hat y\sp{\alpha}$ is a local minimizer of $\hat f\sp{(\alpha)}$.
Then,
by the proximity theorem  (Theorem \ref{THproxintcnv}) 
for $f\sp{(\alpha)}$,
 we obtain
$\|x\sp{\alpha} - x\sp{*} \|_{\infty} \leq  \beta_{n} (\alpha - 1)$,
which shows $x\sp{*} \in \dom f\sp{(\alpha/2)}$.
\end{proof}

In the final phase with $\alpha = 1$,
$f\sp{(\alpha)}$ is an integrally convex function,
and hence, by Theorem \ref{THintcnvlocopt},
an $\alpha$-local minimizer of $f\sp{(\alpha)}$
is a global minimizer of $f\sp{(\alpha)}$.
This observation, with Lemma~\ref{LMoptindomalpha} above,
shows that the output of the algorithm is a global minimizer of $f$.

The complexity of the algorithm can be analyzed as follows.
The number of iterations in the descent method
is bounded by the total number of points in  
$Y = \{ y\in \mathbb{Z}\sp{n} \mid 
     \| \alpha y \|_{\infty} \leq \beta_{n} (2\alpha -1) \}$,
which is bounded by
$(4\beta_{n})\sp{n}$.
For each $y$ 
we examine all of its $3\sp{n}$ neighboring points
to find a descent direction or verify its local minimality.
Thus Step~S1, which updates $x\sp{2\alpha}$ to $x\sp{\alpha}$,
can be done with at most $(12\beta_{n})\sp{n}$ function evaluations.
The number of scaling phases is $\log_{2} K_{\infty}$.
Therefore, the time complexity (or the number of function evaluations)
is bounded by
$(12\beta_{n})\sp{n}\log_{2} K_{\infty}$.
For a fixed $n$, this gives a polynomial bound 
$O(\log_{2} K_{\infty})$ in the problem size.

Finally, we describe how to get rid of the uniqueness assumption
of the minimizer.
Consider a perturbed function
$f_{\varepsilon}(x) 
= f(x) +
\sum_{i=1}\sp{n} \varepsilon\sp{i} x_{i}$
with a sufficiently small $\varepsilon >0$.
By the assumed boundedness of the effective domain of $f$,
the perturbed function has a minimizer, which is unique
as a result of the perturbation.
To find the minimum of $f_{\varepsilon}$ 
it is not necessary to explicitly introduce parameter $\varepsilon$
into the algorithm,  but 
a lexicographically smallest local minimizer $y$ of 
$\hat f(y)$ should be found in Step~S1.

\begin{remark} \rm 
Some technical points are explained here.
By working with $f\sp{(\alpha)}$, we can 
bound the number of iterations for finding 
an $\alpha$-local minimizer
in terms of the number of integer vectors contained in 
$\dom \hat f\sp{(\alpha)}$.
The vector $x\sp{\alpha}$ is an $\alpha$-local minimizer 
for $f\sp{(\alpha)}$, but not necessarily for the original function $f$.
This is why we apply the proximity theorem to $f\sp{(\alpha)}$ 
in the proof of Lemma~\ref{LMoptindomalpha}.
\finbox
\end{remark}

\begin{remark} \rm 
The proximity bound $\beta_{n} (\alpha - 1)$
in Theorem~\ref{THproxintcnv} is linear in $\alpha$.
This linear dependence on $\alpha$ is critical for the 
complexity $O(\log_{2} K_{\infty})$ of the algorithm when $n$ is fixed.
Suppose, for example, that the proximity bound is 
$\beta_{n} (\alpha\sp{m} - 1)$
for some $m >1$. Then 
in the above analysis,
$(2\alpha -1)$ should be replaced by $((2\alpha)\sp{m} - 1)$,
and the total number of points in  
$Y = \{ y\in \mathbb{Z}\sp{n} \mid 
     \| \alpha y \|_{\infty} \leq \beta_{n} ((2\alpha)\sp{m} -1) \}$
is bounded by
$(2\sp{m+1}\beta_{n})\sp{n} \alpha\sp{(m-1)n}$.
The sum of $\alpha\sp{(m-1)n}$ over 
$\alpha = 1, 2, 2\sp{2},\ldots, 2\sp{\lceil \log_{2} K_{\infty} \rceil}$
is of the order of $K_{\infty}\sp{(m-1)n}$.
Then the proposed algorithm will not 
be polynomial in $\log_{2} K_{\infty}$.
Thus the particular form 
 $\beta_{n} (\alpha - 1)$ of our proximity bound
is important for our algorithm.
\finbox
\end{remark}

\section{Concluding Remarks}
\label{SCconrem}

As shown in this paper, 	
the nice properties of ${\rm L}^{\natural}$-convex functions
such as stability under scaling 
and the proximity bound $n(\alpha -1)$ 
are not shared by integrally convex functions in general.
Two subclasses of integrally convex functions
which still enjoy these nice properties
have been introduced in \cite{MMTT17dmpc}
based on discrete midpoint convexity 
(\ref{lnatfmidconv})
for every pair $(x, y) \in \ZZ\sp{n} \times \ZZ\sp{n}$
with $\| x - y \|_{\infty} \geq 2$
or $\| x - y \|_{\infty} = 2$.
Both classes of such functions
are superclasses of ${\rm L}^{\natural}$-convex functions, 
subclasses of integrally convex functions,
and closed under scaling for all $n$
and admit a proximity theorem with the bound $n (\alpha -1)$ 
for all $n$.
See \cite{MMTT17dmpc} for details.

\subsection*{Acknowledgements}
The authors thank Yoshio Okamoto for communicating a relevant reference.
This research was initiated at
the Trimester Program ``Combinatorial Optimization''
at Hausdorff Institute of Mathematics, 2015.
This work was supported by The Mitsubishi Foundation, 
CREST, JST, Grant Number JPMJCR14D2, Japan, and 
JSPS KAKENHI Grant Numbers 26350430, 26280004, 16K00023, 17K00037.


\appendix

\section{An Alternative Proof of Theorem~\ref{THfavtarProp33}}
\label{SCintcnvD2proof}

Here is a proof of Theorem~\ref{THfavtarProp33} 
(local characterization of integral convexity)
that is shorter than the original proof in \cite{FT90}
and valid for functions defined on general integrally convex sets
rather than discrete rectangles.

Obviously, (a) implies (b). 
The proof for the converse,
(b) $\Rightarrow$ (a) , is given by the following two lemmas,
where integral convexity of $\dom f$ and condition {\rm (b)} are assumed.

\begin{lemma} \label{LMboxD2}
Let $B \subseteq \RR\sp{n}$ be a box of size two with integer vertices,
i.e.,
$B = [ {a}, {a} + 2 \bm{1} ]_{\RR}$ for some ${a} \in \ZZ\sp{n}$.
Then $\tilde{f}$ is convex on $B \cap \overline{\dom f}$.
\end{lemma}
\begin{proof}
First, the assumed integral convexity of $\dom f$
implies that 
$B \cap \overline{\dom f} = \overline{B \cap \dom f}$
and that every point in $B \cap \overline{\dom f}$ can be represented 
as a convex combination of points in $B \cap \dom f$.
We may assume $B = [ \bm{0}, 2 \bm{1} ]_{\RR}$.
To prove by contradiction, 
assume that there exist
$x \in B \cap \overline{\dom f}$
and $y\sp{1},\ldots, y\sp{m} \in B \cap \dom f$ such that
\begin{equation} \label{xyifxfyi}
 x =  \sum_{i=1}\sp{m} \lambda_{i} y\sp{i} ,  
\qquad
 \tilde{f}(x) > \sum_{i=1}\sp{m} \lambda_{i} f(y\sp{i}) ,
\end{equation}
where
$\sum_{i=1}\sp{m}  \lambda_{i} = 1$ and 
$\lambda_{i} > 0 \ (i=1,\ldots, m)$.
We may also assume 
$x \in [ \bm{0}, \bm{1} ]_{\RR}$ 
without loss of generality.
For each $j=1,\ldots, n$, we look at the 
$j$-th component of the generating points $y\sp{i}$
to define
\[ 
 I_{j}\sp{0} = \{ i \mid  y\sp{i}_{j} = 0 \},
\qquad
 I_{j}\sp{2} = \{ i \mid  y\sp{i}_{j} = 2 \}.
\] 
Since
$ x_{j} =  \sum_{i=1}\sp{m} \lambda_{i} y\sp{i}_{j} \leq 1$,
if $I_{j}\sp{2} \not= \emptyset$, then
$I_{j}\sp{0} \not= \emptyset$.

Let $j=n$ and suppose that $I_{n}\sp{2} \not= \emptyset$.
Then $I_{n}\sp{0} \not= \emptyset$.
We may assume
$y\sp{1}_{n} = 0$, $y\sp{2}_{n} = 2$; $\lambda_{1} > 0$, $\lambda_{2} > 0$.
By (\ref{intcnvconddist2}) for $(y\sp{1}, y\sp{2})$ 
and the definition of $\tilde{f}$ we have
\[ 
 f(y\sp{1}) + f(y\sp{2})
 \geq 2 \tilde{f}\, \bigg(\frac{y\sp{1} + y\sp{2}}{2} \bigg) 
 = 2 \sum_{k=1}\sp{l} \mu_{k} f(z\sp{k}) ,
\] 
where
\begin{equation} \label{y1y2zk}
\frac{y\sp{1} + y\sp{2}}{2} =  \sum_{k=1}\sp{l}  \mu_{k} z\sp{k},
\qquad
z\sp{k} \in N \bigg( \frac{y\sp{1} + y\sp{2}}{2} \bigg) \cap \dom f
\quad
(k=1,\ldots, l)
\end{equation}
with 
$\mu_{k} > 0$ \  $(k=1,\ldots, l)$ and
$\sum_{k=1}\sp{l}  \mu_{k} = 1$.
This implies, with notation $\lambda = \min(\lambda_{1}, \lambda_{2})$, that 
\[
 \lambda_{1} f(y\sp{1}) + \lambda_{2} f(y\sp{2}) 
\geq
 (\lambda_{1} - \lambda ) f(y\sp{1}) + (\lambda_{2}-\lambda) f(y\sp{2})
 + 2 \lambda \sum_{k=1}\sp{l} \mu_{k} f(z\sp{k}) .
\]
Hence
\[
 \sum_{i=1}\sp{m} \lambda_{i} f(y\sp{i}) 
\geq
 (\lambda_{1} - \lambda ) f(y\sp{1}) + (\lambda_{2}-\lambda) f(y\sp{2})
 + 2 \lambda \sum_{k=1}\sp{l} \mu_{k} f(z\sp{k}) 
+  \sum_{i=3}\sp{m} \lambda_{i} f(y\sp{i}) .
\]
Since  
\[
 x =  (\lambda_{1} - \lambda ) y\sp{1} + (\lambda_{2}-\lambda) y\sp{2}
 + 2 \lambda \sum_{k=1}\sp{l} \mu_{k} z\sp{k}
+  \sum_{i=3}\sp{m} \lambda_{i} y\sp{i} ,
\]
we have obtained another representation of the form 
(\ref{xyifxfyi}).
With reference to this new representation
define
$\hat{I}_{n}\sp{0}$ (resp., $\hat{I}_{n}\sp{2}$)
to be the set of indices of the generators whose $n$-th component is equal to 0 (resp., 2).
Since $z\sp{k}_{n} = 1$ for all $k$ as a consequence of (\ref{y1y2zk})
with
$(y\sp{1}_{n}+ y\sp{2}_{n})/2 = (0 + 2)/2 = 1$,
we have
$\hat{I}_{n}\sp{0} \subseteq I_{n}\sp{0}$,
$\hat{I}_{n}\sp{2} \subseteq I_{n}\sp{2}$
and
$|\hat{I}_{n}\sp{0}| +  |\hat{I}_{n}\sp{2}|  \leq
|I_{n}\sp{0}| +  |I_{n}\sp{2}| - 1$.

By repeating the above process with $j=n$, 
 we eventually arrive at 
a representation of the form of (\ref{xyifxfyi})
with $I_{n}\sp{2} = \emptyset$, which means that 
$y\sp{i}_{n} \in \{ 0,1 \}$ for all generators $y\sp{i}$.

Then we repeat the above process for $j=n-1,n-2, \ldots,1$,
to obtain a representation of the form of (\ref{xyifxfyi})
with $y\sp{i} \in [ \bm{0}, \bm{1} ]_{\ZZ}$
for all generators $y\sp{i}$.
This contradicts the definition of $\tilde{f}$.
\end{proof}

\begin{lemma} \label{LMlineseg}
For any $x, y \in \overline{\dom f}$,
$\tilde{f}$ is convex on the line segment connecting $x$ and $y$.
\end{lemma}
\begin{proof}
Let $L$ denote the (closed) line segment connecting $x$ and $y$,
and consider the boxes $B$, as in Lemma~\ref{LMboxD2}, that intersect $L$.
There exists a finite number of such boxes,
say,
$B_{1}, \ldots, B_{m}$,
and $L$ is covered by the line segments $L_{j} = L \cap B_{j}$
$(j=1,\ldots, m)$.
That is,  $ L = \bigcup_{j=1}\sp{m} L_{j}$.
For each point $z \in L \setminus \{ x, y \}$,
there exists some $L_{j}$ that contains  $z$ in its interior.
Since $L_{j} \subseteq L \subseteq \overline{\dom f}$,
$\tilde{f}$ is convex on $L_{j}$ by Lemma~\ref{LMboxD2}.  Hence%
\footnote{
See  H. Tuy: D.C. optimization: Theory, methods and algorithms,
in: R. Horst and P. M. Pardalos, eds., {\em Handbook of Global Optimization},
Kluwer Academic Publishers, Dordrecht, 1995, 149--216;
Lemma 2 to be specific.
} 
$\tilde{f}$ is convex on $L$.
\end{proof}

\section{Proof of Proposition \ref{PRhilbertbase}}
\label{SChilbertproof}

It is known (cf.~\cite[proof of Theorem 16.4]{Sch86}) that
the set of integer vectors contained in
\[
 F_{A}  = \left\{ 
 \sum_{i \in A} \mu_{i}^{+}(\unitvec{A} + \unitvec{i}) + 
    \sum_{i \in A} \mu_{i}^{-} (\unitvec{A} - \unitvec{i})
 + \sum_{i \in N \setminus A} \mu_{i}^{\circ} (\unitvec{A}+ \unitvec{i})
 + \lambda \unitvec{A} 
    \:\begin{array}{|l}
      \mu_{i}^{+}, \mu_{i}^{-} \in [0,1]_{\RR}\;(i \in A); \\
      \mu_{i}^{\circ} \in [0,1]_{\RR}\; (i \in N \setminus A) ; \\
      \lambda \in [0,1]_{\RR}
      \end{array}\right\}
\]
forms a Hilbert basis of $\tilde{C}_{A}$.
Let $z$ be an integer vector in $F_{A}$.
That is, $z \in \ZZ^{n}$ and 
\begin{align}
 z &=  
 \sum_{i \in A} \mu_{i}^{+}(\unitvec{A} + \unitvec{i})
 + \sum_{i \in A} \mu_{i}^{-} (\unitvec{A} - \unitvec{i})
   + \sum_{i \in N \setminus A} \mu_{i}^{\circ} (\unitvec{A}+ \unitvec{i})  
   + \lambda \unitvec{A} 
\label{hilbasprf0}
 \\
   &= 
   \sum_{i \in A} (\mu_{i}^{+}-\mu_{i}^{-})\unitvec{i} 
  + \sum_{i \in N \setminus A} \mu_{i}^{\circ} (\unitvec{A} + \unitvec{i}) 
  + \left( \lambda + \sum_{i \in A} (\mu_{i}^{+} + \mu_{i}^{-}) \right) \unitvec{A} 
\label{hilbasprf1}
\end{align}
for some 
$\mu_{i}^{+}, \mu_{i}^{-} \in [0,1]_{\RR}\;(i \in A)$;\; 
$\mu_{i}^{\circ} \in [0,1]_{\RR}\; (i \in N \setminus A)$;\; 
$\lambda \in [0,1]_{\RR}$.
Our goal is to show that $z$ can be represented as 
a nonnegative integer combination of vectors in $B_{A}$.

First note that $\mu_{i}^{\circ} \in \{ 0, 1 \}$ for each $i \in N \setminus A$;
define $A^{\circ} = \{ i \in N \setminus A \mid \mu_{i}^{\circ}=1 \}$.
We denote the coefficient of $\unitvec{A}$ in (\ref{hilbasprf1}) as
\[
\xi = \lambda + \sum_{i \in A} (\mu_{i}^{+} + \mu_{i}^{-}) 
\]
and divide into cases according to whether $\xi$ is an integer or not.

Case 1  ($\xi \in \ZZ$):
Using $\xi$ we rewrite  (\ref{hilbasprf1}) as
\begin{align*}
 z  &= 
   \sum_{i \in A} (\mu_{i}^{+}-\mu_{i}^{-})\unitvec{i}
 + \sum_{i \in N \setminus A} \mu_{i}^{\circ} (\unitvec{A} + \unitvec{i}) 
  +  \xi  \unitvec{A} ,
\end{align*}
in which $\xi$ is an integer.
For each $i \in A$,
$\mu_{i}^{+}-\mu_{i}^{-}$ must be an integer, which is equal to $0$, $1$ or $-1$.
Accordingly we define
\begin{align*}
A^{=} 
& = \{ i \in A \mid \mu_{i}^{+}-\mu_{i}^{-} = 0 \} ,
\\
A^{>} 
& = \{ i \in A \mid \mu_{i}^{+}-\mu_{i}^{-} = 1 \}
  = \{ i \in A \mid \mu_{i}^{+}=1, \ \mu_{i}^{-} = 0 \},
\\
A^{<}
& = \{ i \in A \mid \mu_{i}^{+}-\mu_{i}^{-} = -1 \}
  = \{ i \in A \mid \mu_{i}^{+}=0, \ \mu_{i}^{-} = 1 \}
\end{align*}
to rewrite (\ref{hilbasprf0}) as
\begin{equation}\label{hilbasprf2}
 z =  
 \sum_{i \in A^{>}} (\unitvec{A} + \unitvec{i}) + \sum_{i \in A^{<}} (\unitvec{A} - \unitvec{i})
   + \sum_{i \in A^{\circ}} (\unitvec{A}+ \unitvec{i}) 
+ \left( \lambda 
  + \sum_{i \in A^{=}} (\mu_{i}^{+} + \mu_{i}^{-}) \right)\unitvec{A} .
\end{equation}
Here the coefficient of $\unitvec{A}$ is integral, since
\[
 \lambda + \sum_{i \in A^{=}} (\mu_{i}^{+} + \mu_{i}^{-}) 
 = \xi -  \sum_{i \in A^{>} } 1  - \sum_{i \in A^{<}} 1 .
\]
Hence (\ref{hilbasprf2}) gives a representation of $z$ 
as a nonnegative integer combination of vectors in $B_{A}$.

Case 2  ($\xi \not\in \ZZ$):
Let $\eta$ denote the fractional part of $\xi$, i.e.,
$\eta = \xi - \lfloor \xi \rfloor$ with $0 < \eta < 1$.
We rewrite (\ref{hilbasprf1}) as
\begin{equation}\label{hilbasprf4}
 z  =  
    \sum_{i \in A} (\mu_{i}^{+}-\mu_{i}^{-} + \eta)\unitvec{i}
 + \sum_{i \in N \setminus A} \mu_{i}^{\circ} (\unitvec{A} + \unitvec{i}) 
 + \lfloor \xi \rfloor \unitvec{A}.
\end{equation}
For each $i \in A$,
$\mu_{i}^{+}-\mu_{i}^{-} + \eta$ must be an integer, which is equal to $1$ or $0$.
Accordingly we define
\begin{align*}
A^{+} 
& = \{ i \in A \mid \mu_{i}^{+}-\mu_{i}^{-} + \eta = 1 \} ,
\\
A^{-} 
& = \{ i \in A \mid \mu_{i}^{+}-\mu_{i}^{-} + \eta = 0 \}.
\end{align*}
Then  
\begin{equation*}  
\lfloor \xi \rfloor \geq \min (|A^{+}|, |A^{-}| ),
\end{equation*}
which follows from
\begin{align*}
& \mu_{i}^{+} + \mu_{i}^{-} 
\left\{ 
\begin{array}{ll}
 = 2\mu_{i}^{-} + 1-\eta \geq 1 - \eta & (i \in A^{+}) 
\\
=  2\mu_{i}^{+} + \eta \geq \eta & (i \in A^{-}) , 
\end{array} \right.
\\ &
\xi = \lambda + \sum_{i \in A} (\mu_{i}^{+} + \mu_{i}^{-}) 
\geq  (1 - \eta) |A^{+}| + \eta |A^{-}| 
\geq \min (|A^{+}|, |A^{-}| ).
\end{align*}

In the case of $|A^{+}| \leq |A^{-}|$, we see from (\ref{hilbasprf4}) that
\begin{align*}
z &= \sum_{i \in A^{+}} \unitvec{i} 
\phantom{ (\unitvec{A}+ {}) }  
+ \sum_{i \in A^{\circ}} (\unitvec{A}+ \unitvec{i})
+ \lfloor \xi \rfloor \unitvec{A} 
\\
 &=  
    \sum_{i \in A^{+}} (\unitvec{A} + \unitvec{i}) 
   + \sum_{i \in A^{\circ}} (\unitvec{A}+ \unitvec{i})
  +  (\lfloor \xi \rfloor - |A^{+}|) \unitvec{A} ,
\end{align*}
which is a nonnegative integer combination of vectors in $B_{A}$.
In the other case with $|A^{+}| > |A^{-}|$, we have an alternative expression
\begin{align*}
z &= 
 - \sum_{i \in A^{-}} \unitvec{i}
\phantom{ (\unitvec{A} ) }  
 + \sum_{i \in A^{\circ}} (\unitvec{A}+ \unitvec{i})
 + (\lfloor \xi \rfloor+1) \unitvec{A}
\\
 &=  
    \sum_{i \in A^{-}} (\unitvec{A} - \unitvec{i}) 
   + \sum_{i \in A^{\circ}} (\unitvec{A}+ \unitvec{i})
  + (\lfloor \xi \rfloor + 1 - |A^{-}|) \unitvec{A} ,
\end{align*}
which is also a nonnegative integer combination of vectors in $B_{A}$.
This completes the proof of Proposition \ref{PRhilbertbase}.

\end{document}